\documentclass[a4paper,twoside]{article}

\usepackage{amsthm,amssymb,amsmath,booktabs,graphicx,bm,siunitx}
\usepackage[bookmarks=true,bookmarksopen=true,colorlinks=true,citecolor=blue,linkcolor=blue,urlcolor=blue]{hyperref}
\usepackage[dvipsnames]{xcolor}
\usepackage{xspace}
\usepackage{lmodern}

\usepackage[ulem=normalem,draft,authormarkup=none,deletedmarkup=xout]{changes}

\definechangesauthor[name=Mariano, color=ForestGreen]{MM}
\definechangesauthor[name=Thomas, color=Magenta]{TA}
\definechangesauthor[name=Arnd, color=blue]{AR}

\numberwithin{table}{section}
\numberwithin{equation}{section}
\theoremstyle{plain}
\newtheorem{theorem}{Theorem}[section]
\newtheorem{lemma}[theorem]{Lemma}
\newtheorem{corollary}[theorem]{Corollary}

\theoremstyle{definition}
\newtheorem{definition}[theorem]{Definition}
\newtheorem{example}[theorem]{Example}

\theoremstyle{remark}
\newtheorem{remark}[theorem]{Remark}

\newtheorem{assumption}[theorem]{Assumption}

\newcommand{\Pb}{\mbox{\rm (P)}\xspace}
\newcommand{\Pbh}{\mbox{\rm (P$_h$)}\xspace}
\newcommand{\conormal}{n\!}
\newcommand{\tikhonov}{\kappa}  
\newcommand{\mA}{\mathcal{A}}   
\newcommand{\mC}{\mathcal{C}}   
\newcommand{\mD}{\mathcal{D}}   
\newcommand{\mE}{\mathcal{E}}   
\newcommand{\mH}{\mathcal{H}}   
\newcommand{\mI}{\mathcal{I}}   
\newcommand{\mL}{\mathcal{L}}   
\newcommand{\mP}{\mathcal{P}}   
\newcommand{\mQ}{\mathcal{Q}}   
\newcommand{\mS}{\mathcal{S}}   
\newcommand{\mT}{\mathcal{T}}   
\newcommand{\tr}{\textsc{tr}}   
\newcommand{\mZ}{\mathcal{Z}}   

\newcommand{\ps}{{\hat p}}      
\newcommand{\dx}{\,\mathrm{d}x}
\newcommand{\dualH}{H^1(\Omega)'}

\newcommand{\x}{S}              
\newcommand{\Wsj}{\prod_{j=1}^m W^{1/2,2}_{\vec\beta}(\Gamma_j)}

\newcommand{\functional}{J}

\title{Dirichlet control problems with energy regularization governed by non-coercive elliptic equations
\thanks{The second and third author were partially supported by MICIU/AEI/10.13039/501100011033/ under research project PID2023-147610NB-I00.}}

\author{Thomas Apel\thanks{Institute of Mathematics and Computer-Based Simulation. Universit\"at der Bundeswehr M\"unchen, 85577 Neubiberg, Germany, {\tt thomas.apel@unibw.de}. \url{https://orcid.org/0000-0003-3642-3956}}
\and Mariano Mateos\thanks{Departamento de Matem\'{a}ticas, Campus de Gij\'on, Universidad de Oviedo, 33203, Gij\'on, Spain, {\tt mmateos@uniovi.es}.
\url{https://orcid.org/0000-0003-3100-412X}}
\and Arnd R\"osch\thanks{Fakult\"at f\"ur Mathematik, Universt\"at Duisburg-Essen, D-45127 Essen, Germany, {\tt arnd.roesch@uni-due.de.}
\url{https://orcid.org/0009-0001-2163-7153}}
}

\begin{document}
\maketitle
\begin{quote}\textbf{Abstract:}    The present study investigates a linear-quadratic Dirichlet control problem governed by a non-coercive elliptic equation posed on a possibly non-convex polygonal domain. Tikhonov regularization is carried out in an energy seminorm. The regularity of the solutions is established in appropriate weighted Sobolev spaces, and the finite element discretization of the problem is analyzed. 
    In order to recover the optimal rate of convergence in polygonal non-convex domains, graded meshes are required. In addressing this particular problem, it is also necessary to introduce a discrete projection in the sense of $H^{1/2}(\Gamma)$ to deal with the non-homogeneous boundary condition.
    A thorough examination of the approximation properties of the discrete controls reveals that the discrete problems are strongly convex uniformly with respect to the discretization parameter.
    All these ingredients lead to optimal error estimates. Practical computational considerations and numerical examples are discussed at the end of the paper.
\end{quote}

\begin{quote}
\end{quote}
\begin{quote}
\end{quote}

\section{Introduction}\label{S1}
Let $\Omega\subset\mathbb R^2$ be a domain with a polygonal boundary $\Gamma$. 
For some $y_d\in L^2(\Omega)$, $u_d\in H^{1/2}(\Gamma)$, and $\tikhonov > 0$ we are interested in the problem
\[\Pb\qquad\min_{u\in H^{1/2}(\Gamma)} \functional(u) = \frac{1}{2}\Vert y_u-y_d\Vert_{L^2(\Omega)}^2 + \frac{\tikhonov}{2}\vert u-u_d\vert_{H^{1/2}(\Gamma)}^2,\]
where $\vert\cdot\vert_{H^{1/2}(\Gamma)}$ is a seminorm in $H^{1/2}(\Gamma)$ and $y_u\in H^1(\Omega)$ is the solution of the state equation
\begin{equation}\label{E1.1}
  -\nabla\cdot(A(x)\nabla y) +  b(x)\cdot\nabla y + a_0(x) y = 0\text{ in }\Omega,\qquad y=u\text{ on }\Gamma.
\end{equation}
\begin{remark}
  Problems with a source term $f\in H^{-1}(\Omega)$ can be reduced to the form of problem $\Pb$ computing $y_f\in H^1_0(\Omega)$ the unique solution of $-\nabla\cdot(A(x)\nabla y) +  b(x)\cdot\nabla y + a_0(x) y = f\text{ in }\Omega,\ y=0\text{ on }\Gamma$, and redefining $y_d:=y_d-y_f$.
\end{remark}

\medskip
Problem \Pb was first studied in the seminal paper \cite{OfPhanSteinbach2015} for the specific case of the Poisson equation. In that reference $\Omega$ is a convex polygonal domain and $u_d=0$. The seminorm in $H^{1/2}(\Gamma)$ is realized by the seminorm of the harmonic extension $\Vert \nabla\mH u\Vert_{L^2(\Omega)}$, so they study
\begin{align*}
& \min_{u\in H^{1/2}(\Gamma)} \functional(u) = \frac12\Vert y_u - y_d\Vert_{L^2(\Omega)}^2 + \frac{\tikhonov}{2}\Vert \nabla \mH u\Vert_{L^2(\Omega)}^2\\
& -\Delta y = 0\text{ in }\Omega,\ y=u\text{ on }\Gamma.
\end{align*}
The control constrained case is also studied in \cite{OfPhanSteinbach2015}.

Shortly after that, in \cite{Chowdhury_Thirupathi_Nandakumaran_Energy_2017} the following problem was studied:
\begin{align*}
& \min_{z\in H^1(\Omega)} F(z) = \frac12\Vert y_{z_{\vert \Gamma}} - y_d\Vert_{L^2(\Omega)}^2 + \frac{\tikhonov}{2}\Vert \nabla z\Vert_{L^2(\Omega)}^2\\
& -\Delta y = 0\text{ in }\Omega,\ y=z\text{ on }\Gamma.
\end{align*}
The authors of this reference remark that the optimal solution $\bar z$ is a harmonic function, and hence it can be shown that the problems presented in \cite{OfPhanSteinbach2015} and \cite{Chowdhury_Thirupathi_Nandakumaran_Energy_2017} are equivalent; cf. \cite[Remark 2.4]{Chowdhury_Thirupathi_Nandakumaran_Energy_2017}. The control constrained case for the problem studied in \cite{Chowdhury_Thirupathi_Nandakumaran_Energy_2017} was analyzed in \cite{GudiSau2020}.
In these papers, a priori error estimates for a conforming finite element discretization are obtained in convex polygonal domains using quasi-uniform meshes. {\em A posteriori} error estimates in possibly non-convex polygonal domains are also obtained in \cite{Chowdhury_Thirupathi_Nandakumaran_Energy_2017}. Recently, in  \cite{PalGudi2026} the {\em a posteriori} error analysis is extended to problems governed by the more general equation \eqref{E1.1}, provided that the associated bilinear form $\mathfrak{a}(\cdot,\cdot)$---see Section \ref{S2} for the precise definition---is coercive in $H^1_0(\Omega)\times H^1_0(\Omega)$.

The results of \cite{OfPhanSteinbach2015} are extended to a problem governed by the Stokes system in \cite{GongMateosSinglerZhang2022}.
In \cite{Winkler_NM_2020} the governing equation is again the Poisson equation, but now the problem is posed in a possibly non-convex polygonal domain. The discrete approximation is done using a family of quasi-uniform meshes. 

The main novelties of our paper are the following. First, the associated bilinear form $\mathfrak{a}(\cdot,\cdot)$ need not be coercive in $H^1_0(\Omega)\times H^1_0(\Omega)$. Second, we study the discretization using graded meshes and show that for an appropriate grading parameter, we obtain optimal order of convergence. The results obtained in this situation require completely new proofs, techniques and insights. Let us explain them in detail.

In Section \ref{S2} we study the state equation. Since we do not impose coercivity on $\mathfrak{a}(\cdot,\cdot)$, existence and uniqueness of solution of the equation are not immediate. To obtain them we apply appropriate results of our previous works \cite{CMR2020} and \cite{AMR2024}. 

Section \ref{S3} is devoted to the study of the continuous control problem. The proofs of existence and uniqueness of the solution of the optimal control problem in the afore mentioned references use explicitly that the governing equation is the Poisson equation, cf. \cite[Lemma 2.1]{OfPhanSteinbach2015} or the coerciveness of $\mathfrak{a}(\cdot,\cdot)$, cf. \cite[Lemma 2.3]{PalGudi2026}. We are able to show in Lemma \ref{L3.2} that the second derivative of the objective functional is coercive in $H^{1/2}(\Gamma)$, which leads to existence and uniqueness of solution of \Pb under our less restrictive assumptions.

In Section \ref{S4} we exploit the optimality system to deduce the optimal regularity of the solution in weighted Sobolev spaces. The obtained regularity will allow us to prove the optimal order of convergence using graded meshes.

In order to obtain a discretization that allows us solve the problem in practice, we have to discretize in a proper way extension operators $\mE\in \mL(H^{1/2}(\Gamma),H^1(\Omega))$, and particularly the harmonic extension operator. In \cite{Winkler_NM_2020} and \cite{GongMateosSinglerZhang2022}, this is done using the projection $\mP_h\in\mL(H^{1/2}(\Gamma),U_h)$ in the sense of $L^2(\Gamma)$. Here, $U_h$ is the set of discrete controls formed by continuous piecewise linear functions. Unfortunately, the global nature of this projection does not allow us to obtain sufficiently high order error estimates. So we introduce a projection onto $U_h$ in the sense of $H^{1/2}(\Gamma)$. The reader is referred to Section \ref{S5.2} for a discussion of the theoretical and practical implications of this choice. The corresponding discrete harmonic extension is studied in Section \ref{S5.3}. For a thorough study of the properties of the discrete harmonic extension for discrete functions in convex polygonal domains using a quasi-uniform mesh family the reader is referred to \cite{MayRannacherVexler2013}. In our case we have a possibly non-convex domain, graded meshes, and define the discrete extension for every $u\in H^{1/2}(\Gamma)$, so we need to provide new proofs for all our results.

Finally, in Section \ref{S6} we discretize completely the optimal control problem. The reader is referred to Remark \ref{re::R6.2} for a comparison of our discretization with others done in the literature. In Theorem \ref{L6.1} we prove one of our key results: the second derivative of our discrete functional is coercive in $H^{1/2}(\Gamma)$ uniformly with respect to the discretization parameter $h$; see also Remark \ref{re::R6.4}. To obtain the error estimates we insert an intermediate control $u_h^\star$ and apply the coercivity of the second derivative of both the continuous and discrete functionals. Approximations of all the terms lead us to our main result, namely Theorem \ref{th::mainTheorem}, where we obtain  error estimates. Using optimal mesh grading, we prove order $h$ in the energy norm. 

Computational details and two numerical examples are presented in Section \ref{S7} which show the sharpness of  Theorem \ref{th::mainTheorem}.

In \cite{Winkler_NM_2020} higher order error estimates are proved for the quantity $\Vert \bar u -\bar u_h\Vert_{H^{1/2}(\Gamma)}$ provided that the data $y_d$ is a H\"{o}lder function, $u_d=0$ and the domain is convex. This kind of result is not covered by the techniques developed in our paper. 

Besides the afore mentioned references, we also find relevant for the topic the recent papers \cite{ASW2016}, \cite{GongTan2023}, \cite{GanglLoscherSteinbach2025} and \cite{Langeretal2025}. In these papers energy regularization is studied not only for Dirichlet control problems, but also for distributed or Neumann control problems and different discretization schemes are discussed.

\paragraph{Notation} 
Let us introduce notation as in \cite{AMR2024}. Denote by $m$ the number of sides of $\Gamma$ and $\{\x_j\}_{j=1}^m$ its vertices, ordered counterclockwise. For convenience denote also $\x_0=\x_m$ and $\x_{m+1}=\x_1$. We denote by $\Gamma_j$ the side of $\Gamma$ connecting $\x_{j}$ and $\x_{j+1}$, and by $\omega_j\in (0,2\pi)$ the angle interior to $\Omega$ at $\x_j$, i.e., the angle defined by $\Gamma_{j}$ and $\Gamma_{j-1}$, measured counterclockwise. Notice that  $\Gamma_{0}=\Gamma_m$. We use $(r_j,\theta_j)$ as local polar coordinates at $\x_j$, with $r_j=\vert x-\x_j\vert $ and $\theta_j$ the angle defined by $\Gamma_j$ and the segment $[\x_j,x]$. Finally, denote $S:=\{\x_1,\ldots,\x_m\}$.

Throughout the paper we will find several times continuous embeddings $X\hookrightarrow Y$ between pairs of Banach spaces. We will denote $c_i$ a common upper bound of the norms of the these embedding operators.

We will use $y$, $z$ for functions in $H^1(\Omega)$, $u,v$ for functions in $H^{1/2}(\Gamma)$ and Greek letters $\varphi$, $\phi$, $\eta$, $\zeta$ for functions in $H^1_0(\Omega)$. The letters $w,f,g$ will be used for elements that may be in the dual spaces $H^{-1/2}(\Gamma)$, $H^1(\Omega)'$ or $H^{-1}(\Omega)$. We will use calligraphic letters for most of the linear operators.

We will say that $\mE$ is an extension operator if $\mE\in\mL(H^{1/2}(\Gamma),H^1(\Omega))$ is such that the trace of $\mE u$ is $u$ for all $u\in H^{1/2}(\Gamma)$. We will denote $M_{\mE}=\Vert \mE\Vert_ {\mL(H^{1/2}(\Gamma),H^1(\Omega))}$. 
The harmonic extension of $u\in H^{1/2}(\Gamma)$ is the unique function $\mH u \in H^1(\Omega)$ such that
\begin{equation}\label{eq::Hu}
(\nabla\mH u,\nabla \zeta)_\Omega = 0\ \forall \zeta\in H^1_0(\Omega), \quad\mH u = u\text{ on }\Gamma.
\end{equation}
Obviously, $\mH$ is an extension operator.

\begin{remark}\label{re::R1.2}
    Let us notice here that there are several ways to define equivalent norms in $H^{1/2}(\Gamma)$. The norm given by 
\begin{equation}\label{E3.1}\Vert u\Vert_{H^{1/2}(\Gamma)} = \inf\{\Vert z\Vert_ {H^1(\Omega)}:\ z\in H^1(\Omega)\text{ and }z=u\text{ on }\Gamma\}\end{equation}
 is equivalent to $\Vert \mH u\Vert_{H^1(\Omega)}$, to $(\Vert u\Vert ^2_{L^2(\Gamma)}+\Vert \nabla \mH u\Vert ^2_{L^2(\Omega)})^{1/2}$, to the Sobolev-Slobodetskii norm
    \[\left(\Vert u\Vert ^2_{L^2(\Gamma)} + \int_\Gamma\int_\Gamma \frac{(u(x)-u(y))^2}{\vert x-y\vert ^{2}}\dx\mathrm{d}y\right)^{1/2},\]
    or to the norm obtained by real interpolation $(L^2(\Gamma),H^1(\Gamma))_{1/2,2}$. Adopting the definition \eqref{E3.1} implies that the norm of the trace operator $\tr\in \mL(H^1(\Omega),H^{1/2}(\Gamma))$ is $M_{\tr}=1$.
\end{remark}

\section{About the equation}\label{S2}
On $A$, $b$ and $a_0$ we make the following assumptions.

\begin{assumption}\label{A2.1} 
The coefficient functions $A:\bar\Omega\to \mathbb{R}^{2\times 2}$, $b:\bar\Omega\to \mathbb{R}^2$ and $a_0:\bar\Omega\to\mathbb{R}$ are infinitely differentiable on $\bar\Omega\setminus S$ as in \cite[Sect.~6.2.2]{KozlovMazyaRossmann1997}.
The function $A$ satisfies $A=A^T$ and the ellipticity condition
\begin{equation}\label{E2.1}
\exists \Lambda > 0 \text{ such that } \xi\cdot A \xi\geq \Lambda\vert \xi \vert^2\ \ \forall \xi \in \mathbb{R}^2 \text{ and for a.a. } x \in \Omega.
\end{equation}
For the function $a_0$ it is assumed that $a_0(x)\geq 0$ for a.a.\ $x\in\Omega$.
\end{assumption}

\begin{remark}\label{R2.2}
These smoothness assumptions on the coefficients are quite strong. In \cite[Section 3]{AMR2024}, we derived regularity results under the weaker assumptions $A\in C^{0,1}(\bar\Omega)^{2\times 2}$, $b\in L^{\ps}(\Omega)$ with $\ps > 2$ as well as $a_0, \nabla\cdot b\in L^2_{\vec\beta}(\Omega)$ and $b\cdot n\in W^{1/2,2}_{\vec\beta}(\Gamma)$ for some appropriate $\vec\beta$. For a definition of these spaces the reader is referred to Section \ref{S4}.
We conjecture that the results in the current paper can be derived under similarly weak assumptions. In order to avoid further technicalities, we will cite regularity results for regular coefficients from \cite{KozlovMazyaRossmann1997}. 
\end{remark}

\medskip
For every $y\in H^1(\Omega)$, we define $\mA y \in \dualH$ by
\begin{equation}
\langle \mA y,z\rangle_\Omega = \int_\Omega (A \nabla y) \cdot \nabla z \dx + \int_\Omega ( b\cdot\nabla y )z\dx + \int_\Omega a_0 y z\dx\ \forall z\in H^1(\Omega).
\label{E2.6}
\end{equation}
We emphasize that we are not imposing any condition to obtain coercivity in $H^1_0(\Omega)\times H^1_0(\Omega)$ of the associated bilinear form $\mathfrak{a}:H^1(\Omega)\times H^1(\Omega)\to\mathbb{R}$, given by 
\begin{align}\label{eq:def:frak:a}\mathfrak{a}(y,z) := \langle \mA y, z\rangle_\Omega.\end{align}

We denote $\mA_0$ the restriction of $\mA$ to $H^1_0(\Omega)$. Under Assumption \ref{A2.1},  the linear operator $\mA_{0}:H^1_0(\Omega) \longrightarrow H^{-1}(\Omega)$ is an isomorphism; this is \cite[Theorem 2.2]{CMR2020}. Also, $\mA:H^1(\Omega) \longrightarrow \dualH$ is linear continuous; see \cite[Lemma 2.3]{AMR2024}. We will denote $C_{\mA} = \Vert \mA_0^{-1}\Vert_{\mL(H^{-1}(\Omega),H^1_0(\Omega))}$ and $M_{\mA} = \Vert \mA\Vert_{\mL(H^1(\Omega),\dualH)}$. 
We remark that we do not have an isomorphism as an operator in $\mL(H^1(\Omega),\dualH)$ because we are not assuming $a_0>0$ in a set of positive measure.

For every $z\in H^1(\Omega)$ we define $\eta_z\in H^1_0(\Omega)$ as the unique solution of
\begin{equation}\label{eq::defetaRu}
\mathfrak{a}(\eta_{z},\zeta) = \langle -\mA z , \zeta \rangle_\Omega\text{ for all }\zeta\in H^1_0(\Omega).
\end{equation}
We can write $\eta_z = -\mA_{0}^{-1}\mA z$. Therefore, existence and uniqueness of $\eta_z$ is a consequence of \cite[Theorem 2.2]{CMR2020} and \cite[Lemma 2.3]{AMR2024}.
Also, due to the definition of $\mathfrak{a}(\cdot,\cdot)$, equation \eqref{eq::defetaRu} can be written as
\begin{equation}\label{eq::defetaRu2}
\mathfrak{a}(\eta_{z},\zeta) = -\mathfrak{a}(z,\zeta)\text{ for all }\zeta\in H^1_0(\Omega).
\end{equation}
We will say that $y_u\in H^1(\Omega)$ is a weak solution of \eqref{E1.1}  if
\[\mathfrak{a}(y_u,\zeta) = 0\ \forall\zeta\in H^1_0(\Omega),\qquad y_u = u\text{ on }\Gamma.\]

\begin{theorem}\label{T2.2}
For every $u\in H^{1/2}(\Gamma)$ there exists a unique $y_u\in H^1(\Omega)$ weak solution of \eqref{E1.1}. For every extension operator $\mE$, we can write $y_u = \eta_{\mE u} + \mE u$, where $\eta_{\mE u}\in H^1_0(\Omega)$ is the unique solution of \eqref{eq::defetaRu} for $z=\mE u$. There exists $M_{\mS}>0$ independent of $\mE$ such that
\begin{equation}
\label{E::boundyu}
\Vert y_{u}\Vert_{H^1(\Omega)} \leq M_{\mS}\Vert u\Vert_{H^{1/2}(\Gamma)}.
\end{equation}
The continuous linear operator $\mS:H^{1/2}(\Gamma)\to H^1(\Omega)$ given by $\mS u=y_u$ is an extension operator of class $C^2$. For any $v\in H^{1/2}(\Gamma)$, $\mS'(u)v = y_v$ and $\mS''(u) v^2 = 0$.
\end{theorem}
\begin{proof}
 By definition of $\mA$, we have that $-\mA\mE u\in \dualH\hookrightarrow H^{-1}(\Omega)$. By \cite[Theorem 2.2]{CMR2020}, there exists a unique $\eta_{_{\mE u}}\in H^1_0(\Omega)$ such that $\mathfrak{a}(\eta_{_{\mE u}},\zeta) = \langle -\mA\mE u , \zeta \rangle_\Omega$ for all $\zeta\in H^1_0(\Omega)$. Furthermore
   \begin{align}
   \Vert \eta_{_{\mE u}}\Vert_{H^1_0(\Omega)}&\leq C_{\mA} \Vert \mA\mE u\Vert_{H^{-1}(\Omega)} \notag \\
   &\leq C_{\mA} \Vert \mA\mE u\Vert_{\dualH} \leq C_{\mA} M_{\mA} \Vert \mE u\Vert_{H^1(\Omega)}  \leq C_{\mA} M_{\mA} M_{\mE} \Vert u\Vert_{H^{1/2}(\Gamma)}
   \label{E::bound}
   \end{align}
 Take $y_u = \eta_{_{\mE u}} + \mE u$. Then the trace of $y_u$ is $u$ and $\mathfrak{a}(y_u,\zeta) = \mathfrak{a}(\eta_{_{\mE u}},\zeta)+ \mathfrak{a}(\mE u,\zeta) = \langle -\mA\mE u , \zeta \rangle_\Omega + \langle \mA\mE u , \zeta \rangle_\Omega =0$ for all $\zeta\in H^1_0(\Omega)$, and hence $y_u$ is a weak solution of \eqref{E1.1}. The estimate \eqref{E::boundyu} follows from \eqref{E::bound} taking $\mE=\mH$ and $M_{\mS}=M_{\mH}(1+ C_{\mathcal{A}} M_{\mathcal{A}} )$. 
  
  Let us show uniqueness. Assume that $z_u$ is another weak solution. Then $\eta = y_u-z_u\in H^1_0(\Omega)$ and $\mathfrak{a}(\eta,\zeta)=0$ for all $\zeta\in H^1_0(\Omega)$. Using again \cite[Theorem 2.2]{CMR2020} we have that $\eta=0$ and hence $z_u=y_u$.
  
  The last statements are straightforward consequences of the definition of $\mS$.
\end{proof}
\section{Analysis of the control problem}\label{S3}

 Following \cite[p. 726]{OfPhanSteinbach2015}, we will use the seminorm in $H^{1/2}(\Gamma)$ defined by
\begin{equation}\label{eq::defnormh12}\vert u\vert_{H^{1/2}(\Gamma)} : = \Vert  \nabla\mH u\Vert_ {L^2(\Omega)}.\end{equation}
Notice that both this seminorm and the Sobolev-Slobodetskii seminorm are norms in the quotient space $H^{1/2}(\Gamma)/\mathbb{R}$.
\begin{lemma}\label{le::Lemma3.1}
  The seminorms in $H^{1/2}(\Gamma)$ given by 
  \[\vert u \vert_{a} = \Vert  \nabla\mH u\Vert_ {L^2(\Omega)}\text{ and }\vert u \vert_{b} = \left(\int_\Gamma\int_\Gamma \frac{(u(x)-u(y))^2}{\vert x-y\vert ^{2}}\dx\mathrm{d}y\right)^{1/2}
  \]
  are equivalent norms in $H^{1/2}(\Gamma)/\mathbb{R}$ and equivalent seminorms in $H^{1/2}(\Gamma)$.
\end{lemma}
\begin{proof}
  Consider the equivalent norms in $H^{1/2}(\Gamma)$ given by 
\[\Vert u \Vert_{a}^2 = \Vert \mH u\Vert_{H^1(\Omega)}^2=\Vert\mH u\Vert_{L^2(\Omega)}^2 + \vert   u\vert_{a}^2\text{ and }\Vert u \Vert_{b}^2 = \Vert u\Vert_{L^2(\Gamma)}^2 + \vert u \vert_{b}^2.
  \]  
By the equivalence of these norms, we know that there exists $M>0$ such that for every $c\in\mathbb{R}$ we have that
\begin{align*}
 \vert u \vert_{b}  & = \vert u - c \vert_{b} \leq \Vert u - c \Vert_{b} \leq M \Vert u - c\Vert_{a}   =  M \Vert \mH u - c\Vert_{H^1(\Omega)}.
\end{align*}  
Taking the infimum among all $c\in \mathbb{R}$ and applying the Deny-Lions Lemma, we have that
\begin{align*}
 \vert u \vert_{b} \leq M\inf_{c\in\mathbb{R}} \Vert \mH u - c\Vert_{H^1(\Omega)} \leq M' \Vert \nabla \mH u \Vert_{L^2(\Omega)} = M' \vert u \vert_{a}.
\end{align*} 
The other inequality follows in the same way. Since we have not found a proof of the Deny-Lions lemma in $H^{1/2}(\Gamma)$, we provide an elementary proof in Lemma \ref{LB1}.
\end{proof}

\medskip
We can write
\[\functional(u) = \frac{1}{2}\Vert y_u-y_d\Vert_ {L^2(\Omega)}^2+\frac{\tikhonov}{2}\Vert \nabla \mH u-\nabla\mH u_d\Vert_{L^2(\Omega)}^2.\]
Since $\mH u\in H^1_\Delta(\Omega)=\{y\in H^1(\Omega):\ \Delta y\in L^2(\Omega)\}$, for every $u\in H^{1/2}(\Gamma)$ we can define a variational normal derivative $\partial_n \mH u\in H^{-1/2}(\Gamma)$ via
\[
\langle \partial_n \mH u, v\rangle_\Gamma = 
\int_\Omega \nabla\mH u\cdot\nabla \mE v\dx,\ \forall v\in H^{1/2}(\Gamma),
\]
where $\mE$ is an extension operator. The definition is independent of $\mE$: if we consider $\mE_1$ and $\mE_2$ two extension operators, then $(\mE_1-\mE_2)v\in H^1_0(\Omega)$ and hence $(\nabla\mH u,\nabla (\mE_1-\mE_2)v)_\Omega =0$. For the sake of notation, we  define the Dirichlet-to-Neumann or Steklov-Poincar\'{e} operator $\mD:H^{1/2}(\Gamma)\to H^{-1/2}(\Gamma)$ by $\mD u=\partial_n\mH u$. Notice that  $\vert u\vert^2_{H^{1/2}(\Gamma)} = \langle \mD u,u\rangle_\Gamma$ and $\Vert \mD u\Vert_{H^{-1/2}(\Gamma)}\leq M_{\mH}^2\Vert u\Vert_{H^{1/2}(\Gamma)}$ for all $u\in H^{1/2}(\Gamma)$. We will denote $M_{\mD} = M_{\mH}^2$. We can write
\[\functional(u) = \frac{1}{2}\Vert y_u-y_d\Vert_ {L^2(\Omega)}^2+\frac{\tikhonov}{2}\langle \mD (u-u_d),u-u_d\rangle_\Gamma.
\]
For the record, we notice (see \cite[Lemma 2.5]{AMR2024}) that for every $z\in H^1(\Omega)$, $\mathcal{A}^\star z\in \dualH$ can be computed as
\[\langle \mathcal{A}^\star z ,y \rangle_\Omega = \int_\Omega (A^T \nabla z) \cdot \nabla y \dx  - \int_\Omega y\nabla\cdot ( b z)\dx +
\int_\Gamma y z b\cdot n\,\dx
+ \int_\Omega a_0 y z\dx\ \forall y\in H^1(\Omega).
\]
Notice that $\mA^\star_0$ is the restriction to $H^1_0(\Omega)$ of $\mA^\star$, and thus there is no ambiguity.
Since $\mA_{0}$ is an isomorphism between $H^1_0(\Omega)$ and $H^{-1}(\Omega)$, so is $\mA_{0}^\star$. For $g\in H^{-1}(\Omega)$, we define $\phi_g = (\mA^\star_{0})^{-1}g\in H^1_0(\Omega)$. It is the unique solution of 
\begin{equation}\label{eq::Sstar}
\mathfrak{a}(\zeta,\phi_g) = \langle g, \zeta\rangle_\Omega\ \forall \zeta\in H^1_0(\Omega).
\end{equation}
For all $g\in \dualH$, all $v\in H^{1/2}(\Gamma)$ and every extension operator $\mE$, we have that $\mS^\star:\dualH\to H^{-1/2}(\Gamma)$ satisfies that 
\begin{align}
  \langle \mS^\star g, v \rangle_\Gamma &= 
  \langle  g, \mS v \rangle_\Omega 
  =
  \langle  g, \eta_{\mE v} \rangle_\Omega +
  \langle  g, \mE v \rangle_\Omega \notag\\
  &= \mathfrak{a}(\eta_{\mE v},\phi_g)+
  \langle  g, \mE v \rangle_\Omega 
   =  
   - \mathfrak{a}(\mE v,\phi_g) + \langle g, \mE v\rangle_\Omega,
   \label{MME3.3}
\end{align}
where, in the second step, we have used that $\dualH\subset H^{-1}(\Omega)$ and \eqref{eq::Sstar}.
For $g$ regular enough, we show in Lemma \ref{le::L4.1} below that $\partial_{n_A}\phi_g$ is well defined. In this case we can identify $\mS^\star g = -\partial_{n_A}\phi_g$.
For every $u\in H^{1/2}(\Gamma)$, we define the adjoint state $\varphi_u= \phi_{y_u-y_d}\in H^1_0(\Omega)$, the unique solution of 
\[\mathfrak{a}(\zeta,\varphi_u) = (y_u-y_d,\zeta)\ \forall \zeta\in H^1_0(\Omega).\]
For every extension operator $\mE$ we can write
\begin{align}
  (y_u-y_d,y_v ) &= (y_u-y_d,\eta_{\mE v}) + (y_u-y_d,\mE v) =
  \mathfrak{a}(\eta_{\mE v},\varphi_u)+ (y_u-y_d,\mE v)\notag \\
   &= 
  -\mathfrak{a}(\mE v,\varphi_u)+ (y_u-y_d,\mE v). \label{MME3.4}
\end{align}
In order to condense the notation in the next lemma, we will denote 
\begin{align}\label{def:Tw}
    \mT = \mS^\star \mS  + \tikhonov \mD\in \mL(H^{1/2}(\Gamma),H^{-1/2}(\Gamma))\text{ and }w = \mS^\star y_d + \tikhonov \mD u_d\in H^{-1/2}(\Gamma).
\end{align}
Since $M_{\mS^\star} = M_{\mS}$, we have that $M_{\mT} \leq M_{\mS}^2+\tikhonov M_{\mD}$ and
\begin{equation}\label{eq::estw}
    \Vert w\Vert_{H^{-1/2}(\Gamma)} \leq M_{\mS}\Vert y_d\Vert_{L^2(\Omega)} + \tikhonov M_{\mD} \Vert u_d\Vert_{H^{1/2}(\Gamma)}
\end{equation}

\begin{lemma}\label{L3.2}The functional $\functional$ is of class $C^2$ in $H^{1/2}(\Gamma)$. For every $u,v\in H^{1/2}(\Gamma)$ and every extension operator $\mE$, we can write
\begin{align}
  \functional'(u)v  &= (y_u-y_d, y_v) + \tikhonov (\nabla \mH u-\nabla\mH u_d,\nabla \mE v) \notag\\
  &= -\mathfrak{a}(\mE v,\varphi_u)+ (y_u-y_d,\mE v) + \tikhonov (\nabla \mH u-\nabla\mH u_d,\nabla \mE v) \notag\\&= \langle \mT u,v\rangle_\Gamma - \langle  w,v\rangle_\Gamma,\label{OC}\\
  \functional''(u)v^2 &= \Vert y_v\Vert_{L^2(\Omega)}^2 + \tikhonov\Vert \nabla \mH v\Vert_{L^2(\Omega)}^2 = \langle \mT v,v\rangle_\Gamma.\notag
\end{align}
\end{lemma}
\begin{proof}
  The first expression for the derivative of $\functional$ follows from straightforward derivation and the chain rule. The second one from \eqref{MME3.4}. The last one, from \eqref{MME3.3} and the definition of $\mT$ and $w$. The expressions for the second derivative follow from the first and last formulas for $\functional'(u)$.
\end{proof}

\begin{lemma}\label{L3.3}There exists $\nu>0$ such that $\functional''(u)v^2\geq \nu \Vert v\Vert_{H^{1/2}(\Gamma)}^2$ for every $v\in H^{1/2}(\Gamma)$.
\end{lemma}
\begin{proof}
We have to prove the existence of $\nu>0$ such that
\[\Vert \eta_{_{\mH v}}+\mH v\Vert_{L^2(\Omega)}^2+\tikhonov \Vert \nabla\mH v\Vert_{L^2(\Omega)}^2\geq \nu \Vert v\Vert_{H^{1/2}(\Gamma)}^2 \forall v\in H^{1/2}(\Gamma).\]
Suppose that this is false. Then, for every $n\in\mathbb{N}$ there exists $v_n\in H^{1/2}(\Gamma)$ with $\Vert v_n\Vert_ {H^{1/2}(\Gamma)} =1$ such that
\begin{equation}\label{E::contr}
\Vert \eta_{_{\mH v_n}}+\mH v_n\Vert_{L^2(\Omega)}^2+\tikhonov \Vert \nabla\mH v_n\Vert_{L^2(\Omega)}^2\leq \frac{1}{n}.\end{equation}
Since $\{v_n\}_{n=1}^\infty$ is bounded in $H^{1/2}(\Gamma)$, 
there exists a subsequence, which will not be relabeled, and a function $v\in H^{1/2}(\Gamma)$ such that $v_n\rightharpoonup v$ weakly in $H^{1/2}(\Gamma)$. 

For every $g\in \dualH$, $\langle g,\mH v-\mH v_n\rangle_\Omega = \langle \mH^\star g ,v-v_n\rangle_\Gamma$ which tends to zero due to the weak convergence $v_n\rightharpoonup v$. Therefore $\mH v_n \rightharpoonup \mH v$ weakly in $H^1(\Omega)$.

For every $g\in H^{-1}(\Omega)$
\begin{align*}
\langle g,\eta_{_{\mH v_n}}-\eta_{_{\mH v}}\rangle_\Omega &=
\mathfrak{a}(\eta_{_{\mH v_n-\mH v}},\phi_g) \\
&= -\langle \mA (\mH v_n-\mH v), \phi_g\rangle_\Omega = 
-\langle  \mH v_n-\mH v, \mA^\star\phi_g\rangle_\Omega
\end{align*}
Notice that we can see $\mA^*\phi_g$ as an element in $\dualH$, and therefore $\langle g,\eta_{\mH v_n}-\eta_{\mH v}\rangle_\Omega$ tends to zero because $\mH v_n \rightharpoonup \mH v$ weakly in $H^1(\Omega)$. So we have that $\eta_{_{\mH v_n}}\rightharpoonup\eta_{_{\mH v}}$ weakly in $H^1(\Omega)$.

By the Rellich-Kondrachov theorem, $\mH v_n\to \mH v$ and $\eta_{_{\mH v_n}}\to \eta_{_{\mH v}}$ in $L^2(\Omega)$. 
Therefore $\eta_{_{\mH v_n}}+\mH v_n\to \eta_{_{\mH v}}+\mH v$ in $L^2(\Omega)$. 
From \eqref{E::contr}, $\Vert \eta_{_{\mH v_n}}+\mH v_n\Vert_ {L^2(\Omega)}\to 0$, and hence $\eta_{_{\mH v}}+\mH v=0$. 
But this is only possible if $v=0$, because $\eta_{_{\mH v}}\in H^1_0(\Omega)$. 
So $\Vert \mH v_n\Vert_{L^2(\Omega)}\to 0$.
 
 Using again \eqref{E::contr}, we have that $\Vert \nabla \mH v_n\Vert_{L^2(\Omega)}\to 0$. 
 So 
\[\Vert v_n\Vert_{H^{1/2}(\Gamma)} \leq M_{\tr}\Vert \mH v_n\Vert_ {H^1(\Omega)}\to 0.\]
 This contradicts the fact that $\Vert v_n\Vert_ {H^{1/2}(\Gamma)} =1$.
\end{proof}

\begin{theorem}
  Problem \Pb has a unique solution $\bar u\in H^{1/2}(\Gamma)$. Furthermore, there exists a unique $\bar y\in H^1(\Omega)$ and a unique $\bar\varphi\in H^1_0(\Omega)$ such that for every extension operator $\mE$,
\begin{subequations}\label{E3.4}
  \begin{align}
  \label{E3.4a}\mathfrak{a}(\bar y,\zeta) &= 0\ \forall\zeta\in H^1_0(\Omega),\qquad \bar y = \bar u\text{ on }\Gamma, \\
  \label{E3.4b}\mathfrak{a}(\zeta,\bar \varphi) &= (\bar y-y_d,\zeta)\ \forall \zeta\in H^1_0(\Omega),\\ 
  \label{E3.4c} -\mathfrak{a}(\mE v,\bar \varphi) + (\bar y-y_d,\mE v) &+ \tikhonov (\nabla \mH \bar u-\nabla\mH u_d,\nabla \mE v)  = 0\ \forall v\in H^{1/2}(\Gamma).
\end{align}
\end{subequations}
Moreover, we can write
\begin{equation}\label{E3.5}
\langle \mS^\star \mS \bar u + \tikhonov \mD \bar u,v\rangle_\Gamma = \langle  w,v\rangle_\Gamma\ \forall v\in H^{1/2}(\Gamma)
\end{equation}
with $w$ from \eqref{def:Tw}, and the following estimate holds:
\begin{equation}\label{eq::est.u.H12}
   \nu \Vert \bar u \Vert_{H^{1/2}(\Gamma)} \leq  M_{\mS}\Vert y_d\Vert_{L^2(\Omega)} + \tikhonov M_{\mD} \Vert u_d\Vert_{H^{1/2}(\Gamma)}.
\end{equation}
\end{theorem}
\begin{proof}
Since $\functional$ is a quadratic function of $u$, Lemma \ref{L3.3} implies that $\functional$ is strictly convex and coercive.
Since $\functional$ is strictly convex, \Pb has at most one solution. On the other hand, $\functional$ is continuous and convex, and hence weakly lower semicontinuous, so existence of solution follows from the coercivity of $\functional$.

First order optimality conditions follow from the condition $\functional'(\bar u)v=0$ for all $v\in H^{1/2}(\Gamma)$. Uniqueness of $\bar\varphi$ follows from the unique solvability of the adjoint equation. Equation \eqref{E3.5} follows from \eqref{E3.4c} \eqref{OC}.

 Finally, from Lemma \ref{L3.3} and equality \eqref{E3.5}, we deduce that
    \begin{align*}
         \nu\Vert \bar u\Vert_{H^{1/2}(\Gamma)}^2 \leq \langle \mT \bar u,\bar u\rangle_\Gamma = \langle w,\bar u\rangle_\Gamma \leq \Vert w \Vert_{H^{-1/2}(\Gamma)} \Vert \bar u \Vert_{H^{1/2}(\Gamma)},
    \end{align*}
    and \eqref{eq::est.u.H12} follows from \eqref{eq::estw}.
\end{proof}
\begin{remark}\label{R3.2}
The proof of existence of solution of \Pb in the seminal paper \cite{OfPhanSteinbach2015} can be traced to the $H^{1/2}(\Gamma)$ ellipticity of the operator they call $T_\rho$, which is obtained in Lemma 2.1 of that reference using that the second order differential operator is $-\Delta$. In \cite{PalGudi2026}, the proof relies in the result stated in \cite[Lemma 2.3]{PalGudi2026}, which uses in an explicit way the coercivity of the second order differential operator. 
Our proof of coercivity of $\functional$, unlike that of  \cite[Lemma 2.3]{PalGudi2026}, cannot use that $\mathfrak{a}(\cdot,\cdot)$ is coercive in $H^1_0(\Omega)\times H^1_0(\Omega)$.
\end{remark}

\section{Regularity}\label{S4}
To deduce the regularity of the optimal control, we notice that, formally, we can write the first order optimality condition as
\begin{equation}\label{eq:normal:derivative:Hu2}
\tikhonov\partial_n \mH \bar u = \partial_{n_A}\bar\varphi + \tikhonov\partial_n \mH u_d.
\end{equation}
Therefore, the optimal control will inherit the regularity properties of the optimal adjoint state and the target control. We next introduce the definition of the appropriate weighted spaces, show that we can define the conormal derivative of the optimal adjoint state and the normal derivative of the target data and deduce its regularity properties.

For every $j\in\{1,\ldots,m\}$ we will denote by $\lambda_{A,j}$ (resp. $\lambda_{\Delta,j}$) the leading singular exponent
associated to the operator $-\nabla\cdot (A(S_j)\nabla y)$ (resp. $-\Delta y$) in the infinite cone $K_j=\{x\in\mathbb{R}^2:0<r_j,\,0<\theta_j<\omega_j\}$; see \cite[Section 3]{AMR2024}. We define $\lambda_j = \min\{\lambda_{A,j},\lambda_{\Delta,j}\}$ and note that $\lambda_j>\frac12$.

 Let $k\in\mathbb{N}_0$ and $\vec\beta=(\beta_1,\ldots,\beta_m)^T\in\mathbb{R}^m$, $j\in\{1,\ldots,m\}$. For ball-neighborhoods $\Omega_{R_j}$ of $\x_j$ with radius $R_j\leq 1$ and $\Omega^0:=\Omega\setminus\bigcup_{j=1}^m \Omega_{R_j/2}$ we define the spaces $W^{k,2}_{\beta_j}(\Omega_{R_j})$ and $V^{k,2}_{\beta_j}(\Omega_{R_j})$ via their norms
\begin{align*}
  \Vert z\Vert_{W^{k,2}_{\beta_j}(\Omega_{R_j})}^2 = \sum_{\mid\alpha\mid\le k}
  \Vert r_j^{\beta_j}D^\alpha z\Vert_{L^2(\Omega_{R_j})}^2, \quad
  \Vert z\Vert_{V^{k,2}_{\beta_j}(\Omega_{R_j})}^2 = \sum_{\mid\alpha\mid\le k}
  \Vert r_j^{\beta_j-k+\mid\alpha\mid}D^\alpha z\Vert_{L^2(\Omega_{R_j})}^2.
\end{align*}
\begin{remark}\label{rem:VkHk}
Note that $W^{0,2}_{\beta_j}(\Omega_{R_j})=V^{0,2}_{\beta_j}(\Omega_{R_j})=:L^2_{\beta_j}(\Omega_{R_j})$. 
Under the assumption $0 <\beta_j<1$, the equalities $W^{1,2}_{\beta_j}(\Omega_{R_j})=V^{1,2}_{\beta_j}(\Omega_{R_j})$, $W^{2,2}_{\beta_j}(\Omega_{R_j})=V^{2,2}_{\beta_j}(\Omega_{R_j})\oplus\mathcal{P}_0$, and $W^{3,2}_{\beta_j}(\Omega_{R_j})=V^{3,2}_{\beta_j}(\Omega_{R_j})\oplus\mathcal{P}_1$ hold \cite[Theorem 7.1.1]{KozlovMazyaRossmann1997}. The cited theorem states also norm equivalences, e.g.,
\[
\Vert z\Vert_{W^{2,2}_{\beta_j}(\Omega_{R_j})} \sim \Vert z-z(0)\Vert_{V^{2,2}_{\beta_j}(\Omega_{R_j})} +\vert z(0)\vert
\quad\text{for}\quad 0<\beta_j<1.
\]
 We also notice that the spaces $\{z\in W^{1,2}_{\vec\beta}(\Omega):\ z=0\text{ on }\Gamma\}$ and $\{z\in V^{1,2}_{\vec\beta}(\Omega):\ z=0\text{ on }\Gamma\}$ coincide and in this space the norms $\Vert z\Vert_{W^{1,2}_{\vec\beta}(\Omega)}$ and $\Vert z\Vert_{V^{1,2}_{\vec\beta}(\Omega)}$ are equivalent even if $\beta_j=0$ for some $j\in\{1,\ldots,m\}$; see \cite[Lemma 6.6.1]{KozlovMazyaRossmann1997}.
 \end{remark}
 
 Denoting by $\chi_j\in C^\infty(\Omega)$, $j\in\{1,\ldots,m\}$ nonnegative, monotonic, radial cut-off functions with $\chi_j=1$ in $\Omega_{R_j/2}$, $\chi_j=0$ in $\Omega\setminus\Omega_{R_j}$, 
and introducing $\chi_0:=1-\sum_{j=1}^m\chi_j$, the spaces $W^{k,2}_{\vec\beta}(\Omega)$ and $V^{k,2}_{\vec\beta}(\Omega)$ denote the set of all functions $z$ such that
\begin{align*}
  \Vert z\Vert_{W^{k,2}_{\vec\beta}(\Omega)} = \Vert \chi_0 z \Vert_{H^k(\Omega)} + \sum_{j=1}^m
  \Vert \chi_j z\Vert_{W^{k,2}_{\beta_j}(\Omega_{R_j})}, \\
  \Vert z\Vert_{V^{k,2}_{\vec\beta}(\Omega)} = \Vert \chi_0 z\Vert_{H^k(\Omega)} + \sum_{j=1}^m
  \Vert \chi_j z\Vert_{V^{k,2}_{\beta_j}(\Omega_{R_j})},
\end{align*}
respectively, are finite. The corresponding seminorms are defined by setting $\vert \alpha\vert =k$ instead of $\vert \alpha\vert \le k$. 
Remark \ref{rem:VkHk} leads for $k=1$ to the equality $V^{1,2}_{\vec\beta}(\Omega) = W^{1,2}_{\vec\beta}(\Omega)$ if $0<\beta_j<1$ and for $k=2$ to the direct sum
\begin{align}\label{eq:H2V2}
  W^{2,2}_{\vec\beta}(\Omega) = V^{2,2}_{\vec\beta}(\Omega) \oplus \chi_1\mathcal{P}_0\oplus\cdots\oplus\chi_m\mathcal{P}_0
  \quad\text{if}\quad 0<\beta_j<1 \quad \forall j\in\{1,\ldots,m\}
\end{align}
which we will use later. We will also use the notation $L^2_{\vec\beta}(\Omega)$ for $W^{0,2}_{\vec\beta}(\Omega)$.

The trace spaces $W^{k-1/2,2}_{\vec\beta}(\Gamma_j)$, $V^{k-1/2,2}_{\vec\beta}(\Gamma_j)$, $W^{k-1/2,2}_{\vec\beta}(\Gamma)$, and $V^{k-1/2,2}_{\vec\beta}(\Gamma)$ are defined in \cite[Subsections 7.1.3, 6.2.1, 7.1.4]{KozlovMazyaRossmann1997} in a way analogous to \eqref{E3.1}.
e.g., via 
\begin{equation}\label{eq::tracenorm}
    \Vert u\Vert_{W^{k-1/2,2}_{\vec\beta}(\Gamma)}=\inf\{\Vert z\Vert_{W^{k,2}_{\vec\beta}(\Omega)}:z\in W^{k,2}_{\vec\beta}(\Omega),z\vert _{\Gamma\setminus S}=u\}.
\end{equation}
For an equivalent norm, see \cite[Lemma 6.1.2]{KozlovMazyaRossmann1997}. 

\begin{lemma}\label{lem:4.1}
    If $u\in H^{1/2}(\Gamma)$ then $y_u-y_d\in L^2(\Omega)\hookrightarrow L^2_{\vec\beta}(\Omega)\hookrightarrow H^{-1}(\Omega)$ for $\vec\beta\in [\vec0,\vec1)$.
\end{lemma}
\begin{proof}
    We start with $u\in H^{1/2}(\Gamma)$. Then there is a $y_u\in H^1(\Omega)$ with $y_u\vert _\Gamma=u$, see Theorem \ref{T2.2}. Hence $y_u-y_d\in L^2(\Omega)\hookrightarrow L^2_{\vec\beta}(\Omega)\hookrightarrow H^{-1}(\Omega)$ where the latter embedding follows from the definition of the $H^{-1}(\Omega)$-norm, 
    \begin{align*}
        \Vert y_u-y_d\Vert_{H^{-1}(\Omega)} = \sup_{v\in H^1_0(\Omega)} \frac{(y-y_d,v)}{\Vert v\Vert_{H^1_0(\Omega_j)}} \le
        \Vert y-y_d\Vert_{L^2_{\vec\beta}(\Omega)} \sup_{v\in H^1_0(\Omega)} \frac{\Vert v\Vert_{L^2_{-\vec\beta}(\Omega)}}{\Vert v\Vert_{H^1_0(\Omega)}},
    \end{align*} 
    and $\Vert v\Vert_{L^2(\Omega_j)}\le c\Vert v\Vert_{H^1_0(\Omega_j)}$ in the case $\beta_j=0$ and
    $\Vert r_j^{-\beta_j}v\Vert_{L^2(\Omega_j)}\le\Vert v\Vert_{V^{1,2}_{1-\beta_j}(\Omega_j)} \sim\Vert v\Vert_{W^{1,2}_{1-\beta_j}(\Omega_j)}\le c\Vert v\Vert_{H^1_0(\Omega_j)}$ in the case $\beta_j > 0$, see Remark \ref{rem:VkHk}. 
\end{proof}

\begin{lemma}\label{le::L4.1} For any $g\in L^2_{\vec\beta}(\Omega)$, where $1-\lambda_{A,j}<\beta_j<1$, $\beta_j\geq 0$ for all $j\in\{1,\ldots,m\}$, the solution $\phi_g\in H^1_0(\Omega)$ of the dual equation \eqref{eq::Sstar} satisfies that $\phi_g\in V^{2,2}_{\vec\beta}(\Omega)$, $\partial_{\conormal_A} \phi_g \in \prod_{j=1}^m W^{1/2,2}_{\vec\beta}(\Gamma_{j})$ and there exists $c_{\eqref{eq:V2fromL2}}>0$ such that 
    \begin{align}\label{eq:V2fromL2}
        \Vert \phi_g\Vert_{V^{2,2}_{\vec\beta}(\Omega)} +\sum_{j =1}^m \Vert \partial_{\conormal_A} \phi_g \Vert_{W^{1/2,2}_{\vec\beta}(\Gamma_j)}\le c_{\eqref{eq:V2fromL2}}\Vert g\Vert_{L^2_{\vec\beta}(\Omega)}.
    \end{align}
\end{lemma}
\begin{proof}
    We get 
    $\Vert\phi_g\Vert_{V^{2,2}_{\vec\beta}(\Omega)}\le c\Vert g\Vert_{L^2_{\vec\beta}(\Omega)}+\Vert\phi_g\Vert_{V^{1,2}_{\vec\beta}(\Omega)}
    $ from \cite[(6.3.15)]{KozlovMazyaRossmann1997}.
    Using the embedding $L^2_{\vec\beta}(\Omega)\hookrightarrow H^{-1}(\Omega)$, see Lemma \ref{lem:4.1}, we can conclude $\Vert\phi_g\Vert_{V^{1,2}_{\vec\beta}(\Omega)}\sim \Vert\phi_g\Vert_{W^{1,2}_{\vec\beta}(\Omega)} \le c\Vert\phi_g\Vert_{H^1_0(\Omega)}\le c\Vert g\Vert_{H^{-1}(\Omega)}\le c\Vert g\Vert_{L^2_{\vec\beta}(\Omega)}$ such that one part of \eqref{eq:V2fromL2} follows.  In the case $\beta_j=0$ we employed also that $\phi_g=0$ on $\Gamma$; cf. Remark \ref{rem:VkHk}.

    The conormal derivative $\partial_{n_A} \phi$ of a function $\phi\in V^{2,2}_{\vec\beta}(\Omega)$ is defined via
\[\partial_{n_A} \phi=\tr(A\nabla\phi)\cdot n.\]
Its regularity is derived using $\nabla\phi\in V^{1,2}_{\vec\beta}(\Omega)$ to get $A\nabla\phi\in V^{1,2}_{\vec\beta}(\Omega)$ by Assumption \ref{A2.1}, as well as $\tr(A\nabla\phi)\in V^{1/2,2}_{\vec\beta}(\Gamma)$. Noticing that the normal vector $n$ jumps in the corners of the domain we have $\partial_{\conormal_A} \phi \in \prod_{j=1}^m V^{1/2,2}_{\vec\beta}(\Gamma_j) = \prod_{j=1}^m W^{1/2,2}_{\vec\beta}(\Gamma_j).$
 \end{proof}

\begin{remark}
Under the assumptions of the previous lemma it can be proved that $\partial_{n_A}\phi_g \in V^{1/2,2}_{\vec\beta}(\Gamma)$.
This global regularity follows noticing that, since $\phi_g$ is constant on $\Gamma$, then the tangential derivative is zero at every side, and hence, at the corners, we have that the derivatives w.r.t.\ non-colinear vectors are zero. We just sketch the idea of the proof. In the proof of \cite[Lemma A.2]{CasasMateosRaymond2009}, it is shown that if $z\in H^2(\Omega)$ and the tangential derivatives of $z$ vanish then $\partial_{n} z\in H^{1/2}(\Gamma)$. 
The function $z=r^{\beta_j} \phi_g\chi_j$ satisfies these assumptions (note that $r^{\beta_j} \phi_g\chi_j\in H^2(\Omega)$ is equivalent to $\phi_g\chi_j\in V^{2,2}_{\vec\beta}(\Omega)$) hence $\partial_{n}(r^{\beta_j} \phi_g\chi_j)\in H^{1/2}(\Gamma)$ and we have that $r^{\beta_j} \partial_{n}(\phi_g\chi_j)\in H^{1/2}(\Gamma)$. So, for any extension operator, $\mE(r^{\beta_j} \partial_{n}(\phi_g\chi_j))\in H^1(\Omega)$, and we have that $r^{-\beta_j}\mE(r^{\beta_j} \partial_{n}(\phi_g\chi_j))\in V^{1,2}_{\vec\beta}(\Omega)$. Therefore, taking the trace of this function we have that $\partial_{n}(\phi_g\chi_j)\in V^{1/2,2}_{\vec\beta}(\Gamma)$ and consequently $\mE \partial_{n}(\phi_g\chi_j)\in V^{1,2}_{\vec\beta}(\Omega)$. Summing for all $j$, we have that $\mE \partial_{n}\phi_g\in V^{1,2}_{\vec\beta}(\Omega)$ and therefore the trace of this function satisfies that $\partial_{n}\phi_g\in V^{1/2,2}_{\vec\beta}(\Gamma)$. The result is, proved for the Laplace operator. For our operator $A$ the result now follows via the usual technique of coefficient freezing and change of variable thanks to Assumption \ref{A2.1}.
\end{remark}

\medskip
For a regular enough $u$, we have a regularity result for $\mH u$ analogous to that of Lemma \ref{le::L4.1}.
\begin{lemma}\label{th::T4.2}
    For every $u\in W^{3/2,2}_{\vec\beta}(\Gamma)$ with $1-\lambda_{\Delta,j} < \beta_j <1$, $\beta_j\geq 0$ for all $j\in \{1,\ldots,m\}$, 
the harmonic extension $\mH u\in H^1(\Omega)$ of $u$ satisfies that $\mH u\in W^{2,2}_{\vec\beta}(\Gamma)$, $\partial_n\mH u \in \prod_{j=1}^m W^{1/2,2}_{\beta_j}(\Gamma_j)$, and there exists a constant $c_{\eqref{eq::Hureg}}>0$ independent of $u$ such that
    \begin{equation}\label{eq::Hureg}
    \Vert \mH u\Vert_{W^{2,2}_{\vec\beta}(\Omega)} + \sum_{j = 1}^m\Vert \partial_n\mH u\Vert_{W^{1/2,2}_{\beta_j}(\Gamma_j)} \leq c_{\eqref{eq::Hureg}} \Vert u \Vert_{W^{3/2,2}_{\vec\beta}(\Gamma)}.
    \end{equation}
\end{lemma}
\begin{proof}
    We follow \cite[Section 7.3.4]{KozlovMazyaRossmann1997} and seek $\mH u=\mE u+w$ where $w$ satisfies 
    \[
      -\Delta w=\Delta\mE u \quad\text{in }\Omega, \qquad w=0\quad\text{on }\Gamma.
    \]
    Since $\Delta\mE u\in L^2_{\vec\beta}(\Omega)$ we get $w\in V^{2,2}_{\vec\beta}(\Omega)\hookrightarrow W^{2,2}_{\vec\beta}(\Omega)$ and conclude
    \begin{align*}
      \Vert \mH u\Vert_{W^{2,2}_{\vec\beta}(\Omega)}&\le 
      \Vert \mE u\Vert_{W^{2,2}_{\vec\beta}(\Omega)} + \Vert w\Vert_{W^{2,2}_{\vec\beta}(\Omega)}\\&\le 
      \Vert \mE u\Vert_{W^{2,2}_{\vec\beta}(\Omega)} + c\Vert \Delta\mE u\Vert_{L^2_{\vec\beta}(\Omega)}\le 
      c\Vert \mE u\Vert_{W^{2,2}_{\vec\beta}(\Omega)}.
    \end{align*}
    Taking the infimum among all possible extensions, the first part of the result follows from the definition of the $\Vert u\Vert_{W^{3/2,2}_{\vec\beta}(\Gamma)}$-norm.

    The part for the normal derivative follows in the same way as in Lemma \ref{le::L4.1}.
\end{proof}

\medskip
We will assume some extra regularity for the target control.
\begin{assumption}\label{A:u_d}
    The function $u_d$ belongs to $W^{3/2,2}_{\vec\beta}(\Gamma)$ for all $\vec\beta$ such that $\beta_j\in [0,1)$ and $1-\lambda_j < \beta_j$ for all $j\in\{1,\ldots,m\}$.
\end{assumption}

\begin{theorem}\label{th::T4.14}
    If Assumption \ref{A:u_d} holds, then $\bar u \in  W^{3/2,2}_{\beta}(\Gamma)$ for  $1-\lambda_{j}<\beta_j<1$, $\beta_j\geq 0$ for all $j\in\{1,\ldots,m\}$, and
    \begin{equation}
  \Vert \bar u \Vert_{W^{3/2,2}_{\beta}(\Gamma)} \leq c \left(\Vert y_d\Vert_{L^2(\Omega)} + \Vert u_d\Vert_{W^{3/2,2}_{\vec\beta}(\Gamma)}\right)
    \end{equation}
\end{theorem}
\begin{proof}
  Since $u_d\in W^{3/2,2}_{\vec\beta}(\Gamma)\hookrightarrow H^{1/2}(\Gamma)$ we get from \eqref{eq::est.u.H12} that 
  \begin{align}
  \Vert\bar y\Vert_{H^1(\Omega)}\le c\Vert\bar u\Vert_{H^{1/2}(\Gamma)}&\le \notag
  c\left(M_{\mS}\Vert\bar y_d\Vert_{L^2(\Omega)}  + \kappa M_{\mD}\Vert u_d\Vert_{H^{1/2}(\Gamma)}\right) \\\label{eq:proofT4.7} &\le
  c\left( \Vert\bar y_d\Vert_{L^2(\Omega)}  + \Vert u_d\Vert_{W^{3/2,2}_{\vec\beta}(\Gamma)}\right).
  \end{align}
  Lemma \ref{le::L4.1} leads then with $\bar\varphi=\phi_{\bar y-y_d}$ to
  \[
  \sum_{j = 1}^m \Vert \partial_{n_A}\bar\varphi \Vert_{W^{1/2,2}_{\beta_j}(\Gamma_j)} \leq c
  \Vert\bar y-y_d\Vert_{L^2_{\vec\beta}(\Omega)}\le c
  \left( \Vert\bar y_d\Vert_{L^2(\Omega)}  + \Vert u_d\Vert_{W^{3/2,2}_{\vec\beta}(\Gamma)}\right).
  \]
  With \eqref{eq:normal:derivative:Hu2} and Lemma \ref{th::T4.2} we conclude
  \begin{align}
        \sum_{j = 1}^m \Vert \partial_n\mH \bar u \Vert_{W^{1/2,2}_{\beta_j}(\Gamma_j)} &\leq  
        \kappa^{-1}\sum_{j = 1}^m \Vert \partial_{n_A}\bar\varphi \Vert_{W^{1/2,2}_{\beta_j}(\Gamma_j)} +
        \sum_{j = 1}^m \Vert \partial_n\mH \bar u_d \Vert_{W^{1/2,2}_{\beta_j}(\Gamma_j)} \notag\\&\le c
        \left(\Vert y_d\Vert_{L^2(\Omega)} + \Vert u_d\Vert_{W^{3/2,2}_{\vec\beta}(\Gamma)}\right).
        \label{e1}
  \end{align}

   Finally we use the definition of trace norm \eqref{eq::tracenorm}, and the fact that $\mH\bar u$ solves the Neumann problem
    \[-\Delta\mH\bar u + \mH\bar u = \mH\bar u\text{ in }\Omega,\ \partial_n\mH\bar u = \partial_n\mH\bar u\text{ on }\Gamma\]
    to deduce from \cite[Theorem 3.5(c)]{AMR2024} that
    \[\Vert \bar u\Vert_{W^{3/2,2}_{\vec\beta}(\Gamma)} \leq 
    \Vert \mH \bar u\Vert_{W^{2,2}_{\vec\beta}(\Omega)} \leq C 
    \left(\sum_{j =1}^m \Vert \partial_n \mH\bar u\Vert_{W^{1/2,2}_{\vec\beta}(\Gamma_j)}
    + \Vert \mH \bar u\Vert_{L^{2}(\Omega)}\right)
    \]
    Taking into account that $\Vert \mH \bar u\Vert_{L^{2}(\Omega)}\leq M_{\mH}
    \Vert \bar u\Vert_{H^{1/2}(\Gamma)}$, the result follows from \eqref{e1} and \eqref{eq:proofT4.7}.    
\end{proof}

\medskip
We will end this section extending the definition of harmonic extension via the transposition method. First we notice the following result that follows from Lemma \ref{le::L4.1} when the operator is the Laplace operator.
\begin{lemma}\label{th::T4.3}
    For every $g\in L^2(\Omega)$ the unique solution $\psi_g\in  H^1_0(\Omega)$ of
    \[(\nabla\psi_g,\nabla\zeta)_\Omega = (g,\zeta)_\Omega\ \forall \zeta\in H^1_0(\Omega)\]
    satisfies that, for $1-\lambda_{\Delta,j} < \beta_j <1$, $\beta_j\geq 0$ for all $j\in \{1,\ldots,m\}$, $\psi_g\in W^{2,2}_{\vec\beta}(\Omega)$, $\partial_n\psi_g\in \prod_{j=1}^m  W^{1/2,2}_{\vec\beta}(\Gamma_j)$ 
    and there exists $c_{\mathrm{d}}>0$ such that
    \[\Vert \psi_g \Vert_{W^{2,2}_{\vec\beta}(\Omega)} + \sum_{j=1}^m \Vert \partial_n\psi_g \Vert_{W^{1/2,2}_{\vec\beta}(\Gamma_j)}
    \leq c_{\mathrm{d}} \Vert g\Vert_{L^2(\Omega)}.
    \]
\end{lemma}
For $\vec\beta$ satisfying $1-\lambda_{\Delta,j} < \beta_j <1$, $\beta_j\geq 0$ for all $j\in \{1,\ldots,m\}$, we define the harmonic extension of $u\in \left(\prod_{j=1}^m  W^{1/2,2}_{\vec\beta}(\Gamma_j)\right)'$ as the unique $\mH u\in L^2(\Omega)$ such that
\begin{equation}\label{eq::defHarmExtDual}
(\mH u,g)_\Omega = - \langle u, \partial_n \psi_g\rangle_{\left(\prod_{j=1}^m  W^{1/2,2}_{\vec\beta}(\Gamma_j)\right)'\times
\prod_{j=1}^m  W^{1/2,2}_{\vec\beta}(\Gamma_j)}\ \forall g\in L^2(\Omega).
\end{equation}
\begin{theorem}\label{th::T4.4}For every $u\in \left(\prod_{j=1}^m  W^{1/2,2}_{\vec\beta}(\Gamma_j)\right)'$ with $1-\lambda_{\Delta,j} < \beta_j <1$, $\beta_j\geq 0$ for all $j\in \{1,\ldots,m\}$, there exists a unique $\mH u\in L^2(\Omega)$ solution of \eqref{eq::defHarmExtDual} and there exists $M'_{\mH}>0$ such that
\[\Vert \mH u \Vert_{L^2(\Omega)} \leq M'_{\mH} \Vert u \Vert_{\left(\prod_{j=1}^m  W^{1/2,2}_{\vec\beta}(\Gamma_j)\right)'}.\]
If, in addition, $u\in H^{1/2}(\Gamma)$, then $\mH u$ is the harmonic extension of $u$ in the sense of \eqref{eq::Hu}.
\end{theorem}
\begin{proof}
The proof follows the lines of the proof for the case $s=-1/2$ in \cite[Lemma A.6]{CasasMateosRaymond2009}, see eq. (A.17) in that reference. Here, we can use Lemma \ref{th::T4.3} instead of \cite[Lemma A.2]{CasasMateosRaymond2009} which is proved in that reference only for convex domains.

The last statement follows from the identity $(\nabla\mH u,\nabla\psi_g)_\Omega=0$ and integration by parts.
\end{proof}

\section{Numerical approximation. The equations}\label{S5}

\subsection{Discrete spaces and Lagrange interpolation}
Consider a family of  regular triangulations $\{\mathcal{K}_h\}$  graded with mesh grading parameters $\mu_j\in(0,1]$, $j\in\{1,\ldots,m\}$ in the sense of \cite[Section 3.1]{ASW1996}, see also \cite{AMPR2019}.
As usual, $Y_h\subset H^1(\Omega)\cap C(\bar\Omega)$ is the space of continuous piecewise linear functions and $Y_{0,h}=Y_h\cap H^1_0(\Omega)$, $U_h = Y_{h\vert\Gamma}$.
Let $\mI_h:C(\bar\Omega)\to Y_h$ be the Lagrange interpolation operator.

\begin{lemma}\label{L5.1interpolation} 
Let $\vec\beta$ be a vector such that $0\leq \beta_j < 1$ and let $s>0$ be an exponent satisfying 
    \begin{equation}\label{eq:sprime}
        s\leq 1 \quad\text{and}\quad s\leq \frac{1-\beta_j}{\mu_j} \quad\forall j\in\{1,\ldots,m\}.
    \end{equation}
    Then there exists a constant $c_{\mI}>0$, that may depend on $\vec\mu$ but is independent of $h$, such that
    \begin{align}
        \Vert y- \mI_h y\Vert_{H^1(\Omega)} & \leq c_{\mI} h^{s} \Vert y\Vert_{W^{2,2}_{\vec\beta}(\Omega)}\ \forall  y\in W^{2,2}_{\vec\beta}(\Omega),
        \label{eq:interpolation_error}\\
        \Vert u -\mI_h u \Vert_{H^{1/2}(\Gamma)} & \leq M_{\tr} c_{\mI} h^{s} \Vert u\Vert_{W^{3/2,2}_{\vec\beta}(\Gamma)}\ \forall u \in W^{3/2,2}_{\vec\beta}(\Gamma).
        \label{eq:interpolation_error_Gamma}
    \end{align}
\end{lemma}
\begin{proof}
    The estimate \eqref{eq:interpolation_error} is proved in \cite[Lemma 4.1]{AMR2024}. Note that the proof uses, for historical reasons, an intermediate quantity $\vec\lambda>\vec1-\vec\beta$ such that $s<\lambda_j/\mu_j$ for all $j$ is asserted there. The proof adds just local interpolation error estimates and their derivation is completely without $\vec\lambda$ such that the formulation here is more adequate.

    To prove the second estimate we use an extension $\mE u\in W^{2,2}_{\vec\beta}(\Omega)\hookrightarrow H^1(\Omega)$. Observing that $\mI_h u=\tr(\mI_h\mE u)$ and using \eqref{eq:interpolation_error} we obtain
    \begin{align*}
        \Vert u -\mI_h u \Vert_{H^{1/2}(\Gamma)} & \le M_{\tr} \Vert \mE u -\mI_h\mE u \Vert_{H^1(\Omega)} 
        \le  M_{\tr} c_{\mI} h^{s} \Vert \mE u\Vert_{W^{2,2}_{\vec\beta}(\Omega)}.
    \end{align*}
    Taking the infimum among all possible extensions, estimate \eqref{eq:interpolation_error_Gamma} follows from the definition of $\Vert u\Vert_{W^{3/2,2}_{\vec\beta}(\Gamma)}$.
\end{proof}

\subsection{Projection onto $U_h$ in the sense of $H^{1/2}(\Gamma)$} \label{S5.2}
In order to discretize the problem, it is advisable to use some quasi-interpolation operator $\mQ_h \in \mL(H^{1/2}(\Gamma),U_h)$. It is desirable that it satisfies the following properties. First, it should be a projection operator, i.e., $\mQ_h u_h=u_h$ for all $u_h\in U_h$; this is essential to be able to get the final contradiction in the proof of Theorem \ref{L6.1} and to start the proof of Lemma \ref{le::L6.9}. Second, 
we should be able to obtain a high enough order of convergence for error estimates in dual norms for functions in $H^{1/2}(\Gamma)$. 

In \cite{Winkler_NM_2020} and  \cite{GongMateosSinglerZhang2022} the projection $\mP_h$ onto $U_h$ in the sense of $L^2(\Gamma)$ is used. But the global nature of this operator does not allow us to obtain all necessary error estimates in weighted spaces and with graded meshes. The order that we would obtain in Theorem \ref{T5.9} is only $s=0.5$. The order of convergence obtained using a quasi-uniform mesh family in the analogous result \cite[Theorem 1]{Winkler_NM_2020} is $s=1-\max\{\beta_j\}$. The orthogonality properties are satisfied in the wrong space. We need them in $H^{1/2}(\Gamma)$ but they are satisfied only in $L^2(\Gamma)$.
The Scott-Zhang quasi-interpolation operator $\Pi_h$ is a projection operator and its local nature allows to obtain error estimates analogous to that of Corollary \ref{C5.4}, but we do not know of any improved error estimate in dual norms. 
The Carstensen interpolation operator $\mC_h$ also exhibits a local nature that has made it useful to discretize Dirichlet problems in graded meshes; see \cite{ApelNicaisePfefferer:16}. But in general $\mC_h u_h\neq u_h$, so we cannot use this either.

We will use a projection onto $U_h$ in the sense of $H^{1/2}(\Gamma)$. We will show that this projection has the desired properties and it will allow us to obtain the error estimates for the discretization of the problem. Notice that although this projection is not easy to construct from a practical point of view, in practice we will never have to compute it, since, contrary to a pure boundary value problem, the boundary data $u$ are not prescribed but $u_h$ is an unknown.

Consider the bilinear form
\[\mathfrak{p}(u,v) = ( u,v)_\Gamma + (\nabla \mH u,\nabla \mH v)_\Omega.\]
This bilinear form is obviously continuous and, due to Lemma \ref{le::Lemma3.1}, it is also coercive in $H^{1/2}(\Gamma)$: 
There exist constants $0<\Lambda_{\mathfrak{p}}<M_{\mathfrak{p}}$ such that 
\[\Lambda_{\mathfrak{p}} \Vert u\Vert_{H^{1/2}(\Gamma)}^2 \leq  \mathfrak{p}(u,u)\text{ and } \mathfrak{p}(u,v) \leq M_{\mathfrak{p}} \Vert u\Vert_{H^{1/2}(\Gamma)}\Vert v\Vert_{H^{1/2}(\Gamma)}\ \forall u,v\in H^{1/2}(\Gamma).\] The following lemma is an elementary  consequence of this fact.

\begin{lemma}
For every $u\in H^{1/2}(\Gamma)$ there exists a unique $\mQ_h u\in U_h$ such that
\[\mathfrak{p}(\mQ_h u, v_h) = \mathfrak{p}(u,v_h)\ \forall v_h\in U_h.\]
Furthermore, if $u_h\in U_h$, then $\mQ_h u_h = u_h$. It is $\mQ_h\in \mL(H^{1/2}(\Gamma),U_h)$ and, for $M_{\mQ} = M_{\mathfrak{p}}/\Lambda_{\mathfrak{p}}$, 
\begin{equation}\label{eq::MQ}
\Vert \mQ_h u\Vert_{H^{1/2}(\Gamma)}\leq M_{\mQ}\Vert u\Vert_{H^{1/2}(\Gamma)}\ \forall u\in H^{1/2}(\Gamma),
\end{equation}
and
\begin{equation}\label{eq::LA1}
\Vert u-\mQ_h u\Vert_{H^{1/2}(\Gamma)} \leq M_{\mQ} \inf_{u_h\in U_h}\Vert u-u_h\Vert_{H^{1/2}(\Gamma)}\ \forall  u\in H^{1/2}(\Gamma).\end{equation}
\end{lemma}

Error estimates in the norms $L^2(\Gamma)$ and $H^{1/2}(\Gamma)$ are not needed in this paper, but may be of independent interest; we add them in Appendix \ref{sec:appA}.
To obtain an optimal error estimate in $\left(\Wsj\right)'$, we will need the following regularity result about a dual problem.

\begin{theorem}\label{L5.5}
  For every $g\in H^{-1/2}(\Gamma)$ there exists a unique solution $w_g\in H^{1/2}(\Gamma)$ of the dual problem
  \begin{equation}\label{eq::A.01}\mathfrak{p}(v,w_{g}) = \langle g, v\rangle_{\Gamma}\ \forall v\in H^{1/2}(\Gamma).\end{equation}
  If, furthermore, $g\in \Wsj$, for $\vec\beta$ such that $1-\lambda_{\Delta,j}< \beta_j < 1$ and $\beta_j\geq 0$ for all $j\in\{1,\ldots,m\}$, then $w_g\in W^{3/2,2}_{\vec\beta}(\Gamma)$  and there exists a constant $C_{\mathfrak{p}}>0$ such that
  \begin{equation}\label{eq::A.02}\Vert w_g\Vert_{W^{3/2,2}_{\vec\beta}(\Gamma)} \leq C_{\mathfrak{p}} \sum_{j=1}^m\Vert g\Vert_{W^{1/2,2}_{\vec\beta}(\Gamma_j)}.\end{equation}
\end{theorem}
\begin{proof}
  Since $\mathfrak{p}(\cdot,\cdot)$ is a scalar product in $H^{1/2}(\Gamma)$, from Riesz's representation theorem or Lax-Milgram's theorem, we know that there exists a unique $w_{g}\in H^{1/2}(\Gamma)$ solution of \eqref{eq::A.01}.
  
  If $g\in \Wsj \hookrightarrow H^{-1/2}(\Gamma)$, then, using the definition of $\mathfrak{p}(v,w_g)$, we have that $\mH w_g$ satisfies
  \[(\nabla \mH v,\nabla \mH w_g)_\Omega +(v,w_g)_\Gamma= (g, v)_\Gamma\ \forall v\in H^{1/2}(\Gamma).\]
  Consider $z\in H^1(\Omega)$, denote $v=z_{\vert \Gamma}$ and define $\eta = z-\mH v\in H^1_0(\Omega)$. Since $(\nabla z,\nabla\mH w_g)_\Omega  = (\nabla \eta,\nabla\mH w_g)_\Omega  +(\nabla \mH v,\nabla\mH w_g)_\Omega = (\nabla \mH v,\nabla\mH w_g)_\Omega$, we have that $\mH w_g$ solves the Robin problem
  \[(\nabla z,\nabla \mH w_g)_\Omega +(z_{\vert \Gamma},w_g)_\Gamma= (g, z_{\vert \Gamma})_\Gamma\ \forall z\in H^{1}(\Omega).\] 
  We can hence apply the regularity result in \cite[Theorem 3.5(c)]{AMR2024} and we have that $\mH w_g\in W^{2,2}_{\vec{\beta}}(\Omega)$ and there exists a constant $\tilde C>0$ such that
  $\Vert \mH w_g\Vert_{W^{2,2}_{\vec{\beta}}(\Omega)} \leq \tilde C \sum_{j=1}^m\Vert g\Vert_{W^{1/2,2}_{\vec\beta}(\Gamma_j)}$ and the result follows from the definition of trace norm.
\end{proof}
\begin{theorem}\label{T5.9}
   Let  $u$ belong to $H^{1/2}(\Gamma)$ and consider  $s\leq 1$, $s<\lambda_{\Delta,j}/\mu_j$ for all $j\in\{1,\ldots,m\}$. Then there exists a constant $\tilde c >0$, that may depend on $\vec\mu$ but is independent of $u$ and $h$, such that
  \[\Vert u-\mQ_h u\Vert_{\left(\Wsj\right)'} \leq \tilde c h^{s} \Vert u-\mQ_h u\Vert_{H^{1/2}(\Gamma)},\]
  for any $\vec\beta$ such that $\beta_j\in[0,1)$ and $s \leq (1-\beta_j)/\mu_j <\lambda_{\Delta,j}/\mu_j$ for all $j\in\{1,\ldots,m\}$.
\end{theorem}
\begin{proof}
The vector $\vec\beta$ satisfies the assumptions of Theorem \ref{L5.5}. Hence, for any $g\in \Wsj$, there exists a unique solution $w_g\in W^{3/2,2}_{\vec\beta}(\Gamma)$ of $\mathfrak{p}(v,w_{g})=(g,v)_\Gamma$ for all $v\in H^{1/2}(\Gamma)$. For the sake of notation, denote $G=\{g\in \Wsj:\ \sum_{j=1}^m\Vert g\Vert_{W^{1/2,2}_{\vec\beta}(\Gamma_j)} = 1\}$.
  
  Using the definition of $w_g$, the definition of $\mQ_h u$, the continuity of the bilinear form $\mathfrak{p}(\cdot,\cdot)$, the interpolation error estimate \eqref{eq:interpolation_error_Gamma}, and Theorem \ref{L5.5}, we obtain
  \begin{align*}
    \Vert &u-\mQ_h u\Vert_{\left(\Wsj\right)'}  = \sup_{g\in G} (u-\mQ_h u,g)_\Gamma  = \sup_{g\in G} \mathfrak{p}(u-\mQ_h u,w_g)\\
    & =  \sup_{g\in G} \mathfrak{p}(u-\mQ_h u,w_g-\mI_h w_g) \leq M_{\mathfrak{p}}\sup_{g\in G} \Vert u-\mQ_h u\Vert_{H^{1/2}(\Gamma)}  \Vert w_g-\mI_h w_g\Vert_{H^{1/2}(\Gamma)} \\
    & \leq M_{\mathfrak{p}} c_{\mI}  h^s \Vert u-\mQ_h u\Vert_{H^{1/2}(\Gamma)}  \sup_{g\in G} \Vert w_g\Vert_{W^{3/2,2}_{\vec\beta}(\Gamma)} \\
    & \leq
     M_{\mathfrak{p}} c_{\mI} C_{\mathfrak{p}} h^s \Vert u-\mQ_h u\Vert_{H^{1/2}(\Gamma)}  \sup_{g\in G} \sum_{j=1}^m\Vert g\Vert_{W^{1/2,2}_{\vec\beta}(\Gamma_j)} \\
    & =      M_{\mathfrak{p}} c_{\mI} C_{\mathfrak{p}} h^s  \Vert u-\mQ_h u\Vert_{H^{1/2}(\Gamma)},
  \end{align*}
and  the result follows for $\tilde c = M_{\mathfrak{p}} c_{\mI} C_{\mathfrak{p}}$.
\end{proof}

\subsection{The discrete harmonic extension}\label{S5.3}$ $

\begin{definition}
We will say that $\mE_h$ is a discrete extension operator if $\mE_h\in\mL(H^{1/2}(\Gamma),H^1(\Omega))$ is such that $\mE_h u\in Y_h$ and $(\mE_h u)_{\vert\Gamma} = \mQ_h u$ for all $u\in H^{1/2}(\Gamma)$. 
\end{definition}

\medskip
Notice that $\mE_h u_h=u_h$ on $\Gamma$ for all $u_h\in U_h$ and $\mE_h u = \mE_h \mQ_h u$ for all $u\in H^{1/2}(\Gamma)$.
\begin{remark}\label{re::R5.9}
    In \cite[eq. (5.4)]{GongMateosSinglerZhang2022} a discrete extension operator using the $L^2(\Gamma)$-projection is proposed.
\end{remark}

\medskip
For $u\in H^{1/2}(\Gamma)$, the discrete harmonic extension $\mH_h u\in Y_h$ is the unique solution of 
\begin{equation}
    \label{eq::discHarmext}
    (\nabla \mH_h u,\nabla \zeta_h)_\Omega = 0\quad \forall \zeta_h\in Y_{0,h},\qquad \mH_h u = \mQ_h u\text{ on }\Gamma.
\end{equation}

For any discrete extension operator $\mE_h$, we will denote $M_{\mE_h} = \Vert \mE_h\Vert_ {\mL(H^{1/2}(\Gamma),H^1(\Omega))}$. 
Notice that in general $M_{\mE_h}$ may depend on $h$, but $M_{\mH_h}$ is bounded independently of $h$: 
\begin{lemma}
    \label{Le::contHh}There exists a constant $M_{\mH}^\star>0$, that may depend on $\vec\mu$ but is independent of $h$, such that
    \begin{equation}\label{eq::discrHarmCont}
        \Vert \mH_h u\Vert_{H^1(\Omega)}\leq M_{\mH}^\star \Vert u\Vert_{H^{1/2}(\Gamma)}\ \forall u\in H^{1/2}(\Gamma).
    \end{equation}
\end{lemma}
\begin{proof}
Denote $u_h = \mQ_h u\in U_h$. We first prove \eqref{eq::discrHarmCont} for $u_h\in U_h$, and the result will then follow from estimate \eqref{eq::MQ}.

    We will estimate first $\Vert\nabla\mH_h u_h\Vert_{L^2(\Omega)}$. Let us denote $\Pi_h:H^1(\Omega)\to Y_h$ the Scott--Zhang interpolation operator; see \cite{ScottZhang1990}. We have that $\Pi_h \mH u_h\in Y_h$ and that $\Pi_h \mH u_h = u_h$ on $\Gamma$, so $\mH_hu_h-\Pi_h \mH u_h\in Y_{0,h}$. From the general stability result \cite[Theorem 3.1]{ScottZhang1990} we can also deduce the existence of a constant $M_{\Pi}>0$ independent of $h$ such that $\Vert \Pi_h y\Vert_{H^1(\Omega)}\leq M_{\Pi}\Vert y\Vert_{H^1(\Omega)}$ for all $y\in H^1(\Omega)$. Following the argument in the proof of \cite[Lemma 3.2(i)]{MayRannacherVexler2013} we can write
    \begin{align*}
        \Vert\nabla\mH_h u_h\Vert_{L^2(\Omega)}^2 & = (\nabla\mH_h u_h, \nabla(\mH_hu_h-\Pi_h \mH u_h))_\Omega + (\nabla\mH_h u_h, \nabla\Pi_h \mH u_h)_\Omega \\
        & \leq \Vert\nabla\mH_h u_h\Vert_{L^2(\Omega)} \Vert\nabla\Pi_h \mH u_h\Vert_{L^2(\Omega)} \\ & \leq M_{\Pi} M_{\mH} \Vert\nabla\mH_h u_h\Vert_{L^2(\Omega)} \Vert u_h \Vert_{H^{1/2}(\Gamma)}
    \end{align*}
and therefore
\begin{equation}\label{eq::stabHh1}
    \Vert\nabla\mH_h u_h\Vert_{L^2(\Omega)} \leq M_{\Pi} M_{\mH} \Vert u_h \Vert_{H^{1/2}(\Gamma)}
\end{equation}
    To estimate $\Vert\mH_h u_h\Vert_{L^2(\Omega)}$ we deduce from Lemma \ref{th::T4.3} the existence of $\psi\in W^{2,2}_{\vec\beta}(\Omega)\cap H^1_0(\Omega)$, unique solution of
    \[-\Delta\psi = \mH_h u_h\text{ in }\Omega,\ \psi = 0\text{ on }\Gamma,\]
    where $\vec\beta$ satisfies $\beta_j\geq 0$ and $1-\lambda_{\Delta,j} < \beta_j < 1$ for all $j\in\{1,\ldots,m\}$. Using integration by parts, the formula \eqref{eq::defHarmExtDual}, the fact that $\mI_h\psi$ is well defined and belongs to $Y_{0,h}$, the interpolation error estimate \eqref{eq:interpolation_error}, Lemma \ref{th::T4.3}, estimate \eqref{eq::stabHh1}, and the continuity of the continuous harmonic extension, we have
    \begin{align*}
        \Vert & \mH_h u_h\Vert_{L^2(\Omega)}^2 = (\mH_h u_h,-\Delta\psi)_\Omega
         = (\nabla\mH_h u_h,\nabla\psi)_\Omega - (u_h,\partial_n\psi)_\Gamma \\
         & = (\nabla\mH_h u_h,\nabla(\psi-\mI_h\psi ))_\Omega + (\mH u_h,\mH_h u_h)_\Omega\\ & \leq c_{\mathrm{d}} c_{\mI} M_{\Pi}M_{\mH}  h^{s} \Vert u_h\Vert_{H^{1/2}(\Gamma)} \Vert \mH_h u_h\Vert_{L^2(\Omega)} + M_{\mH} \Vert u_h\Vert_{H^{1/2}(\Gamma)} \Vert \mH_h u_h\Vert_{L^2(\Omega)} 
    \end{align*}
     for some $s\leq 1$ such that $s \leq (1-\beta_j)/\mu_j$ for all $j\in\{1,\ldots,m\}$. Therefore
     \begin{equation}\label{eq::boundHhuhL2} \Vert\mH_h u_h\Vert_{L^2(\Omega)} \leq (c_{\mathrm{d}} c_{\mI} M_{\Pi}M_{\mH} \mathrm{diam}(\Omega)^{s} + M_{\mH}) \Vert u_h\Vert_{H^{1/2}(\Gamma)}.
     \end{equation}
Gathering estimates \eqref{eq::boundHhuhL2} and \eqref{eq::stabHh1} with \eqref{eq::MQ}, the  result follows for \[M_{\mH}^\star = M_{\mQ}( c_{\mathrm{d}} c_{\mI} M_{\Pi}M_{\mH} \mathrm{diam}(\Omega)^{s} + M_{\mH} + M_{\Pi}M_{\mH})\]
and the proof is complete.
\end{proof}
\begin{remark}\label{re::R5.10}
    Since $M_{\mH}\leq M_{\mH}^\star$, abusing notation we will rename $M_{\mH}:= M^\star_{\mH}$  and we will use $M_{\mH}$ for both constants.
\end{remark}
\begin{remark}
    Our proof of Lemma \ref{Le::contHh} follows closely that of \cite[Lemma 3.2]{MayRannacherVexler2013}. Notice, nevertheless, that our proof is slightly different since we do not obtain a bound of $\Vert\mH_h u_h\Vert_{L^2(\Omega)}$ in terms of  $\Vert u_h\Vert_{L^{2}(\Gamma)}$, as it is done in \cite[Lemma 3.2(ii)]{MayRannacherVexler2013}. To obtain this estimate, an inverse inequality  is used in this reference but cannot be used with graded meshes. See also \cite[Theorem 3.4.6]{VexlerMeidner2025book}. 
\end{remark}    

\begin{lemma}[C\'ea type argument] There exists a constant $c_{\textsc{cea}}>0$, that may depend on $\vec\mu$ but is independent of $h$, such that
\begin{align}\label{eq::Cea}
    \Vert \nabla(\mH u-\mH_h u)\Vert_{L^2(\Omega)} \le 
     c_{\textsc{cea}}\Vert  \mH u-y_h\Vert_{H^1(\Omega)} \quad\forall u\in H^{1/2}(\Gamma)\text{ and } y_h\in Y_h.
\end{align}
\end{lemma}
\begin{proof}
    The peculiarity of the proof is that $\mH u\not=\mH_h u$ on $\Gamma$. Nevertheless the Galerkin orthogonality 
\begin{align*}
    (\nabla(\mH u-\mH_h u),\nabla\chi_h)=0 \quad\forall\chi_h\in Y_{0,h}
\end{align*}
holds, and hence with the standard proof the C\'ea lemma
\begin{align*}
    \Vert \nabla(\mH u-\mH_h u)\Vert_{L^2(\Omega)} \le 
    \Vert \nabla(\mH u-z_h)\Vert_{L^2(\Omega)} \quad
    \forall z_h\in Y_h \text{ such that } z_h=\mQ_h u\text{ on }\Gamma.
\end{align*}
For any $y_h\in Y_h$ with $\tr y_h = v_h\in U_h$, we define now $z_h=y_h+\mH_h(\mQ_h u-v_h)$, so that $z_h=\mQ_h u$ on $\Gamma$, and obtain 
\begin{align*}
    \Vert \nabla(\mH u-z_h)\Vert_{L^2(\Omega)} &\le \Vert \nabla(\mH u-y_h)\Vert_{L^2(\Omega)} +
    \Vert \nabla  \mH_h(\mQ_h u-v_h)\Vert_{L^2(\Omega)}
\end{align*}
The second term can be further estimated by using \eqref{eq::discrHarmCont} and C\'{e}a's inequality \eqref{eq::LA1} 
\begin{align*}
    \Vert \nabla & \mH_h(\mQ_h u-v_h)\Vert_{L^2(\Omega)} \le 
    \Vert  \mH_h(\mQ_h u-v_h)\Vert_{H^1(\Omega)} \leq M_{\mH}
    \Vert  \mQ_h u- v_h\Vert_{H^{1/2}(\Gamma)} \\ &\le
    M_{\mH} \big(\Vert   u-\mQ_h u\Vert_{H^{1/2}(\Gamma)} + \Vert   u- v_h\Vert_{H^{1/2}(\Gamma)} \big) \\
    &\leq 
    M_{\mH}(1+M_{\mQ})\Vert   u- v_h\Vert_{H^{1/2}(\Gamma)} \le  M_{\mH}(1+M_{\mQ}) M_{\tr}\Vert  \mH u-y_h\Vert_{H^1(\Omega)}
\end{align*}
such that we obtain \eqref{eq::Cea} for $c_{\textsc{cea}} = 1+M_{\mH}(1+M_{\mQ}) M_{\tr}$.
\end{proof}

\begin{corollary}\label{C5.8} If $u\in W^{3/2,2}_{\vec\beta}(\Gamma)$ for $\vec\beta$ such that $1-\lambda_{\Delta,j}< \beta_j < 1$ and $\beta_j\geq 0$ for all $j\in\{1,\ldots,m\}$, then for all $s\leq 1$ such that $s<\lambda_{\Delta,j}$ for all $j\in\{1,\ldots,m\}$, there exists a constant $C_1>0$, that may depend on $\vec\mu$ but is independent of $u$ and $h$, such that
    \begin{equation}\label{eq::err_discharm}
        \Vert\nabla\mH u -\nabla \mH_h u\Vert_{L^2(\Omega)} \leq C_1 h^{s} \Vert u\Vert_{W^{3/2,2}_{\vec\beta}(\Gamma)}.
    \end{equation}
\end{corollary}
\begin{proof}
    Using $y_h=\mI_h\mH\bar u$ in \eqref{eq::Cea}, the interpolation error estimate \eqref{eq:interpolation_error}, and Theorem \ref{th::T4.2}, we see that
\begin{align*}
\Vert\nabla\mH u -\nabla \mH_h u\Vert_{L^2(\Omega)} & \leq
    c_{\textsc{cea}} \Vert \mH\bar u-\mI_h\mH\bar u \Vert_{H^1(\Omega)} \le c_{\textsc{cea}} c_{\mI} h^{s} \Vert\mH\bar u\Vert_{W^{2,2}_{\vec\beta}(\Omega)} \\
    & \leq c_{\textsc{cea}} c_{\mI} c_{\eqref{eq::Hureg}} h^{s} \Vert\bar u\Vert_{W^{3/2,2}_{\vec\beta}(\Gamma)},
\end{align*}
and the result follows for $C_1 = c_{\textsc{cea}} c_{\mI} c_{\eqref{eq::Hureg}}$.
\end{proof}

\begin{lemma}\label{L5.9} 
 Let  $u$ belong to $H^{1/2}(\Gamma)$ and consider  $s\leq 1$ such that $s<\lambda_{\Delta,j}/\mu_j$ for all $j\in\{1,\ldots,m\}$. Then, there exists a constant $C_0 >0$, that may depend on $\vec\mu$ but is independent of $u$ and $h$, such that
      \begin{equation}\label{eq::err:L2Harm}
        \Vert\mH u -\mH_h u\Vert_{L^2(\Omega)} \leq C_0 h^{s} \Vert u\Vert_{H^{1/2}(\Gamma)}.
    \end{equation}
\end{lemma}
\begin{proof}
By the triangle inequality, we have that
\[
\Vert\mH u -\mH_h u\Vert_{L^2(\Omega)} \leq \Vert\mH u -\mH \mQ_h u\Vert_{L^2(\Omega)}+\Vert\mH \mQ_h u -\mH_h u\Vert_{L^2(\Omega)}.
\]
  Take $\vec\beta$ such that $\beta_j\in[0,1)$ and $s <(1-\beta_j)/\mu_j <\lambda_{\Delta,j}/\mu_j$ for all $j\in\{1,\ldots,m\}$.

On one hand, from Theorems \ref{th::T4.4} and \ref{T5.9}, we have that
\begin{align*}
\Vert\mH u -\mH \mQ_h u\Vert_{L^2(\Omega)} & 
 \leq  M'_{\mH} \Vert u-\mQ_h u\Vert_{\left(\Wsj\right)'}
 \\
& \leq M'_{\mH}  \tilde c h^{s} \Vert u-\mQ_h u\Vert_{H^{1/2}(\Gamma)} \leq  M'_{\mH} \tilde c  (1+M_{\mQ}) h^{s} \Vert u\Vert_{H^{1/2}(\Gamma)}.
\end{align*}

    To estimate the second term we use a duality argument. From Lemma \ref{th::T4.3} that there exists a unique $\psi\in H^1_0(\Omega)\cap W^{2,2}_{\vec\beta}(\Omega)$ such that 
\[
(\nabla\psi,\nabla\zeta)_\Omega=(\mH \mQ_hu -\mH_h u,\zeta)_\Omega \text{ for all  }\zeta\in H^1_0(\Omega).
\]
Noting that $\mH \mQ_h u -\mH_h u\in H^1_0(\Omega)$, that $\mI_h\psi\in Y_{0,h}\subset H^1_0(\Omega)$ together with the definition of harmonic and discrete harmonic extension, the continuity properties of these operators, the interpolation error estimate \eqref{eq:interpolation_error}, and Lemma \ref{th::T4.3}, we obtain
\begin{align*}
    \Vert \mH \mQ_hu -\mH_h u\Vert_{L^2(\Omega)}^2&=
 (\nabla(\mH \mQ_h u -\mH_h u),\nabla \psi)= 
    (\nabla(\mH \mQ_h u -\mH_h u),\nabla (\psi-\mI_h \psi))\\&\le
  \Vert\nabla(\mH \mQ_h u -\mH_h u)\Vert_{L^2(\Omega)}
    \Vert\nabla (\psi-\mI_h \psi)\Vert_{L^2(\Omega)} \\&\le
    2 c_{\mI} M_{\mH} M_{\mQ} h^{s}\Vert u \Vert_{H^{1/2}(\Omega)} \Vert \psi\Vert_{W^{2,2}_{\vec\beta}(\Omega)} \\
    &\leq
    2 c_{\mI} M_{\mH} M_{\mQ} c_{\mathrm{d}} h^{s}\Vert u \Vert_{H^{1/2}(\Omega)} \Vert \mH \mQ_hu -\mH_h u\Vert_{L^2(\Omega)},
\end{align*}
and the result follows for $C_0 = M'_{\mH} \tilde c  (1+M_{\mQ}) + 2 c_{\mI} M_{\mH} M_{\mQ} c_{\mathrm{d}}$.
\end{proof}

\medskip
Besides the error estimates, we will also need convergence results for the low regularity case to prove the key result Theorem \ref{L6.1}.
\begin{lemma}\label{co::conv::discharm}For every $u\in H^{1/2}(\Gamma)$,
$
  \mH_h u\to \mH u\text{ strongly in }H^1(\Omega)
$.
\end{lemma}
\begin{proof}
The convergence $\Vert \mH u - \mH_h u \Vert_{L^2(\Omega)}\to 0$ follows from Lemma \ref{L5.9}. The convergence $\Vert \nabla(\mH u - \mH_h u) \Vert_{L^2(\Omega)}\to 0$
    results from a density argument as in \cite[Theorem 18.2]{Ciarlet91}. We can use this argument since we have proved the C\'{e}a type inequality \eqref{eq::Cea} and we have already proved error estimates for more regular $u$.
\end{proof}

\begin{lemma}\label{MML5x}
  For every $g\in \dualH$, $\mH^\star_h g\to \mH^\star g$ strongly in $H^{-1/2}(\Gamma)$.
  \end{lemma}
\begin{proof}
  Denote $\hat S = \{v\in H^{1/2}(\Gamma):\ \Vert v\Vert _{H^{1/2}(\Gamma)} = 1\}$. For $\tilde g\in L^2(\Omega)$, we can write, using \eqref{eq::err:L2Harm},
  \begin{align*}
    \Vert \mH^\star \tilde g - \mH_h^\star\tilde g \Vert _{H^{-1/2}(\Gamma)}   
    &= \sup_{v\in \hat S}\langle \mH^\star \tilde g - \mH_h^\star\tilde g,v\rangle_\Gamma 
    = \sup_{v\in \hat S}\langle \tilde g, \mH v- \mH_h v\rangle_\Gamma  \\
    &= \sup_{v\in \hat S}( \tilde g, \mH v- \mH_h v)_\Gamma  
    \leq  \sup_{v\in \hat S}\Vert \tilde g\Vert _{L^2(\Omega)} \Vert\mH v -\mH_h v\Vert_{L^2(\Omega)} \\
    & \leq \sup_{v\in \hat S}\Vert \tilde g\Vert _{L^2(\Omega)} C h^{s} \Vert v\Vert _{H^{1/2}(\Gamma)},
  \end{align*}
  and we obtain that
  \begin{equation}\label{eq::conv::reg}
    \Vert \mH^\star \tilde g - \mH_h^\star\tilde g \Vert _{H^{-1/2}(\Gamma)} \to 0\text{ as }h\to 0.
  \end{equation}
  Take $g\in \dualH$. By the triangle inequality, and taking into account Remark \ref{re::R5.10}, we have, for any $\tilde g\in L^2(\Omega)$, that
\begin{align}
    \Vert &\mH^\star g - \mH_h^\star g \Vert _{H^{-1/2}(\Gamma)} \notag \\ &\leq
    \Vert \mH^\star g - \mH^\star \tilde g \Vert _{H^{-1/2}(\Gamma)} +
    \Vert \mH^\star \tilde g - \mH_h^\star \tilde g \Vert _{H^{-1/2}(\Gamma)} \notag  +
    \Vert \mH_h^\star \tilde g - \mH_h^\star g \Vert _{H^{-1/2}(\Gamma)}\notag \\
    &\leq 2 M_{\mH}\Vert  g - \tilde g \Vert _{H^{-1/2}(\Gamma)} +
    \Vert \mH^\star \tilde g - \mH_h^\star \tilde g \Vert _{H^{-1/2}(\Gamma)}.\label{eq::triang}
\end{align}
Since $L^2(\Omega)$ is dense in $\dualH$, for every $\varepsilon>0$ there exists $\tilde g_\varepsilon \in L^2(\Omega)$ such that $2 M_{\mH}\Vert  g - \tilde g_\varepsilon \Vert _{H^{-1/2}(\Gamma)} < \varepsilon/2$. From \eqref{eq::conv::reg}, we deduce the existence of $h_\varepsilon > 0 $ such that $\Vert \mH^\star \tilde g_\varepsilon - \mH_h^\star \tilde g_\varepsilon \Vert _{H^{-1/2}(\Gamma)} < \varepsilon/2$ for all $0< h < h_\varepsilon$. Therefore, from \eqref{eq::triang}, we have that $\Vert \mH^\star \tilde g - \mH_h^\star \tilde g \Vert _{H^{-1/2}(\Gamma)} < \varepsilon$ for all $0< h < h_\varepsilon$ and the proof is complete.
\end{proof}

\subsection{State approximation}
As we say in Theorem \ref{T2.2}, the state $y_u$ can be written as the sum of $\eta_{\mE u}$ and $\mE u$ for any extension operator $\mE$. We have already studied the approximation of $\mE u$ when we use the harmonic extension. We next investigate the discretization of $\eta$.

Let $\eta\in H^1_0(\Omega)$ and $\eta_h\in Y_{0,h}$ be the solutions of the problems
\begin{align}
\mathfrak{a}(\eta,\zeta)& =  \langle f,\zeta\rangle_\Omega\ \forall\zeta\in H^1_0(\Omega) \label{eq::E5.4.1}\\
\mathfrak{a}(\eta_h,\zeta_h) &=  \langle f,\zeta_h\rangle_\Omega\ \forall\zeta_h\in Y_{0,h}\label{eq::E5.4.2}
\end{align}
with the bilinear form $\mathfrak{a}:H^1(\Omega)\times H^1(\Omega)\to\mathbb{R}$ defined in \eqref{eq:def:frak:a}. 

\begin{lemma}\label{L5.1}
  There exists $h_{\mA}>0$ such that for every $0<h<h_{\mA}$ and every $f\in H^{-1}(\Omega)$ there exist a unique solution $\eta_h\in Y_{0,h}$ of \eqref{eq::E5.4.2}
   and a constant $c_{\mathcal{A}}$ independent of $\vec\mu$ and $h$ such that
  \[\Vert \eta_h\Vert_{H^1_0(\Omega)} \leq c_{\mathcal{A}} \Vert f\Vert_{H^{-1}(\Omega)}.\]
\end{lemma}
\begin{proof}
    A similar result (homogeneous Dirichlet bundary conditions) is proved in \cite[Lemma 3.1]{CMR2021} but for $f\in L^2(\Omega)$, a convex domain and quasi-uniform mesh family. On the other hand, an analogous result for a Neumann problem with a right hand side in $\dualH$ was obtained for possibly non-convex polygonal domains and graded meshes in \cite[Theorem 4.2]{AMR2024}. The proof of Lemma \ref{L5.1} follows the lines of \cite[Theorem 4.2]{AMR2024} with the corresponding modifications that are meaningful using the $W^{2,2}_{\vec\beta}(\Omega)$ regularity results obtained in Section \ref{S4}.
\end{proof}

\begin{lemma}
    If $f\in H^{-1}(\Omega)$, then there exists a constant $c_{\eqref{eq::erretaL2}}>0$, that may depend on $\vec\mu$ but is independent of $h$, such that 
    \begin{equation}\label{eq::erretaL2}
    \Vert \eta - \eta_h\Vert_{L^2(\Omega)} \leq c_{\eqref{eq::erretaL2}} h^s \Vert \eta-\eta_h \Vert_{H^1_0(\Omega)},
    \end{equation}
    where $s\leq 1$ satisfies $s < \lambda_{A,j}/\mu_j$ for all $j\in\{1,\ldots,m\}$.

    If, further, $f\in L^2_{\vec\beta}$, where $\vec\beta$ satisfies $\beta_j\geq 0$ and $1-\lambda_{A,j} < \beta_j < 1$ for all $j\in\{1,\ldots,m\}$, then there exists a constant $c_{\eqref{eq::erreta}}>0$, that may depend on $\vec\mu$ but is independent of $h$
    \begin{equation}\label{eq::erreta}
        \Vert\eta-\eta_h\Vert_{H^1_0(\Omega)} \leq c_{\eqref{eq::erreta}} h^s \Vert f\Vert_{L^2_{\vec\beta}(\Omega)},
    \end{equation}
   where $s\leq 1$ satisfies $s \leq (1-\beta_j)/\mu_j$ for all $j\in\{1,\ldots,m\}$. 
\end{lemma}
\begin{proof}
Estimate \eqref{eq::erreta} is proved in \cite[Theorem 3.6]{CMR2021} for $f\in L^2(\Omega)$, a convex domain and a quasi-uniform mesh family. An analogous result for a Neumann problem with a right hand side in $L^2_{\vec\beta}(\Omega)$, a possibly non-convex domain, and graded meshes is proved in \cite[Theorem 4.5]{AMR2024}. The proof of \eqref{eq::erreta} follows the same lines noting that $\Vert \eta\Vert_{W^{2,2}_{\vec\beta}(\Omega)}\le c_{\eqref{eq:V2fromL2}}\Vert f\Vert_{L^2_{\vec\beta}(\Omega)}$; see Lemma \ref{le::L4.1}.

To prove \eqref{eq::erretaL2} we define $\phi=\phi_{\eta-\eta_h}\in H^1_0(\Omega)$ the unique solution of the dual problem
\[\mathfrak{a}(\zeta,\phi) = (\eta-\eta_h,\zeta)\ \forall \zeta\in H^1_0(\Omega).\]
Thanks to Lemma \ref{le::L4.1}, we know that $\phi\in V^{2,2}_{\vec\beta}(\Omega)$ for some $\vec\beta$ satisfying $\beta_j\in [0,1)$ and  $1-\beta_{j} < \lambda_{A,j}$ for all $j\in\{1,\ldots,m\}$. Noting that $V^{2,2}_{\vec\beta}(\Omega)\hookrightarrow W^{2,2}_{\vec\beta}(\Omega)$, and denoting $c_i$ the norm of the embedding operator, we conclude with the help of the Poincar\'{e} inequality, the interpolation error estimate \eqref{eq:interpolation_error} and estimate \eqref{eq:V2fromL2} that
\begin{align*}
   \Vert \eta - \eta_h\Vert_{L^2(\Omega)}^2 & = \mathfrak{a}(\eta-\eta_h,\phi) = \mathfrak{a}(\eta-\eta_h,\phi-\mI_h\phi) \\
   & \leq M_{\mA} \Vert \eta-\eta_h \Vert_{H^1(\Omega)} \Vert \phi-\mI_h\phi \Vert_{H^1(\Omega)} \\ & \leq
   c_{\mI} M_{\mA}(1+c_\Omega)  h^s \Vert \eta-\eta_h \Vert_{H^1_0(\Omega)} \Vert \phi\Vert_{W^{2,2}_{\vec\beta}(\Omega)} \\
   & \leq c_{i} c_{\mI} M_{\mA} (1+c_\Omega) h^s \Vert \eta-\eta_h \Vert_{H^1_0(\Omega)} \Vert \phi\Vert_{V^{2,2}_{\vec\beta}(\Omega)}\\
   & \leq c_{\eqref{eq:V2fromL2}} c_{i} c_{\mI} M_{\mA} (1+c_\Omega) h^s \Vert \eta-\eta_h \Vert_{H^1_0(\Omega)} \Vert \eta - \eta_h\Vert_{L^2(\Omega)},
\end{align*}
where $s\leq 1$ satisfies $s \leq (1-\beta_j)/\mu_j$ for all $j\in\{1,\ldots,m\}$.
\end{proof}

\medskip
For every $z\in H^1(\Omega)$, we have that $\mA z\in \dualH\hookrightarrow H^{-1}(\Omega)$. By Lemma \ref{L5.1}, there exists a unique solution $\eta_h(z)\in Y_{0,h}$  of
  \begin{equation}\label{E5.1}\mathfrak{a}(\eta_h(z),\zeta_h) = - \langle \mA z,\zeta_h\rangle_\Omega\ \forall\zeta_h\in Y_{0,h}.\end{equation}
\begin{lemma}
  There exists a constant $c_{\eqref{eq::erretaL2_bis}}>0$ such that for every $z\in H^1(\Omega)$
  \begin{equation}\label{eq::erretaL2_bis}
    \Vert \eta_z-\eta_h(z) \Vert_{L^2(\Omega)} \leq c_{\eqref{eq::erretaL2_bis}} h^s \Vert z \Vert_{H^1(\Omega)},
  \end{equation}
  where $s\leq 1$ satisfies $s < \lambda_j/\mu_j$ for all $j\in\{1,\ldots,m\}$ and $\eta_z\in H^1_0(\Omega)$ is defined in \eqref{eq::defetaRu}.
\end{lemma}
\begin{proof}
    The result is an immediate consequence of \eqref{eq::erretaL2}, the fact that $\mA_{0}:H^1_0(\Omega)\mapsto H^{-1}(\Omega)$ is an isomorphism and Lemma \ref{L5.1}, so we can write $c_{\eqref{eq::erretaL2_bis}} = 2 c_{\eqref{eq::erretaL2}} c_{\mA} M_{\mA}$.
\end{proof}

\medskip
For every $u\in H^{1/2}(\Gamma)$ we will denote $y_h(u)\in Y_h$ the solution of
\begin{equation}\label{eq::discrete_state_eqn}
\mathfrak{a}(y_h(u),\zeta_h)=0\ \forall \zeta_h\in Y_{0,h},\ 
y_h(u) \equiv \mQ_h u,
\text{ on }\Gamma.\end{equation}
Moreover, we will denote the corresponding discrete solution operator by $\mS_h: u\mapsto y_h(u)$ as solution of
\eqref{eq::discrete_state_eqn}.

\begin{theorem}\label{th::T5.17}There exists $h_{\mA} > 0$ such that for all $0<h < h_{\mA}$ and every $u\in H^{1/2}(\Gamma)$ there exists a unique $y_h(u)\in Y_h$ solution of \eqref{eq::discrete_state_eqn}. Furthermore, for every $u\in H^{1/2}(\Gamma)$ and every discrete extension operator $\mE_h$, we can write $y_h(u)= \eta_h(\mE_h u) + \mE_h u$, where $\eta_h(\mE_h u)\in Y_{0,h}$ is defined in \eqref{E5.1}. Finally, there exists $M_{\mS}^{\star}>0$, that may depend on $\vec\mu$ but is independent of $u$ and $h$, such that
  \begin{equation}\label{E5.5}\Vert y_h(u)\Vert_{H^1(\Omega)} \leq  M_{\mS}^{\star} \Vert u\Vert_ {H^{1/2}(\Gamma)}.\end{equation}
The discrete solution operator $\mS_h:H^{1/2}(\Gamma)\to H^1(\Omega)$ is a discrete extension operator whose norm is bounded independently of $h$.  
\end{theorem}
\begin{proof}
On one hand, since $\eta_h(\mE_h u)\in Y_{0,h}$ and by our definition of discrete extension, $\mE_h u=\mQ_h u$ on $\Gamma$, it is clear that $\eta_h(\mE_h u) + \mE_h u=\mQ_h u$ on $\Gamma$. Also, by \eqref{E5.1}, 
  \[\mathfrak{a}\big(\eta_h(\mE_h u) + \mE_h u,\ \zeta_h\big) = - \langle \mA \mE_h u,\zeta_h\rangle_\Omega + \mathfrak{a}(\mE_h u,\zeta_h) =0\]
  and hence $y_h(u)= \eta_h(\mE_h u) + \mE_h u$ solves \eqref{eq::discrete_state_eqn}. If $z_h\in Y_h$ is a solution of \eqref{eq::discrete_state_eqn}, then the difference $\eta_h=y_h(u)-z_h\in Y_{0,h}$ satisfies $\mathfrak{a}(\eta_h,\zeta_h)=0$ for all $\zeta_h\in Y_{0,h}$ and by Lemma \ref{L5.1} we have $\eta_h=0$.
  To obtain an estimate independent of $h$, we use as discrete extension the discrete harmonic extension. From Lemma \ref{L5.1} and using that both the norm of this extension operator as well as the norm of $\mQ_h$ are independent of $h$ we obtain
  \[\Vert y_h(u)\Vert_{H^1(\Omega)}  = \Vert  \eta_h(\mH_h u) + \mH_h u\Vert_{H^1(\Omega)} \leq   M_{\mH}( 1 + c_{\mA} M_{\mA})\Vert u\Vert_{H^{1/2}(\Gamma)}.\]
The final statement follows directly from the previous estimate.
\end{proof}

\begin{remark}\label{re::R5.18}
The constant $M_{\mS}^\star$ obtained in Theorem \ref{th::T5.17} has the same formula as the one obtained at the end of the proof of Theorem \ref{T2.2}, but they may be different because of our abuse of notation warned in Remark \ref{re::R5.10}. We abuse notation again and use $M_{\mS} = \max\{M_{\mS},M_{\mS}^\star\}$ for both constants.
\end{remark}

\begin{lemma}\label{le::stateerror}
    There exist $h_{\mA} > 0$ and  a constant $c_{\eqref{eq:.sterrest}}>0$, that may depend on $\vec\mu$ but is independent of $h$, such that for all $0<h < h_{\mA}$
    and for every $u\in W^{3/2,2}_{\vec\beta}(\Gamma)$, where $\vec\beta$ satisfies $\beta_j\geq 0$ and $1-\lambda_j < \beta_j<1$ for all $j\in\{1,\ldots,m\}$, the following estimate holds:
    \begin{equation}\label{eq:.sterrest}\Vert y_u - y_h(u) \Vert_{L^2(\Omega)} \leq c_{\eqref{eq:.sterrest}} h^s \Vert  u \Vert_{W^{3/2,2}_{\vec\beta}(\Gamma)},\end{equation}
    where $s\leq 1$ satisfies $s \leq (1-\beta_j)/\mu_j$ for all $j\in\{1,\ldots,m\}$.
\end{lemma}
\begin{proof}
For every $g\in L^2(\Omega)$ and with the homogenization technique, we get
\begin{align*}
   & ( y_u- y_h(u),g)_\Omega  = (\eta_{_{\mH u}} -  \eta_h(\mH_h u), g)_\Omega+
  ( \mH u -  \mH_h u, g)_\Omega \\
  & = ( \eta_{_{\mH u}} -  \eta_h(\mH u), g)_\Omega+
   ( \eta_h(\mH u) -  \eta_h(\mH_h u), g)_\Omega+
  ( \mH u -  \mH_h u, g)_\Omega = \mathrm{I} + \mathrm{II} + \mathrm{III}
\end{align*}
To estimate I, we notice that $s\leq (1-\beta_j)/\mu_j < \lambda_j/\mu_j \leq \lambda_{A,j}/\mu_j$, and hence we can use \eqref{eq::erretaL2_bis} to obtain
\begin{align*}
    \mathrm{I} & = ( \eta_{_{\mH u}} -  \eta_h(\mH u), g)_\Omega\leq \Vert \eta_{_{\mH u}} -  \eta_h(\mH u) \Vert_{L^2(\Omega)} \Vert g \Vert_{L^2(\Omega)} \\
    & \leq C h^s \Vert \mH u \Vert_{H^1(\Omega)} \Vert g \Vert_{L^2(\Omega)}
    \leq c_{\eqref{eq::erretaL2_bis}} M_{\mH} h^s \Vert  u \Vert_{H^{1/2}(\Gamma)} \Vert g \Vert_{L^2(\Omega)}\\
    & \leq c_i  c_{\eqref{eq::erretaL2_bis}} M_{\mH} h^s \Vert  u \Vert_{W^{3/2,2}_{\vec\beta}(\Gamma)} \Vert g \Vert_{L^2(\Omega)}
\end{align*}
where we used the embedding $W^{3/2,2}_{\vec\beta}(\Gamma)\hookrightarrow H^{1/2}(\Gamma)$ in the last step.

To estimate II we use the stability estimate for $\Vert \eta_h\Vert_{H^1_0(\Omega)}$ provided in Lemma \ref{L5.1}, the Poincar\'{e} inequality, the continuity of $\mA$. We also notice that the relations among $\beta_j$, $\lambda_j$ and $s$ together with the inequality $\lambda_j\leq \lambda_{\Delta,j}$ imply that we can use estimates \eqref{eq::err_discharm} and \eqref{eq::err:L2Harm} to obtain
\begin{align*}
  \mathrm{II} & =  ( \eta_h(\mH u) -  \eta_h(\mH_h u), g)_\Omega \leq 
  \Vert \eta_h(\mH u-\mH_h u)\Vert_{L^2(\Omega)} \Vert g\Vert_{L^2(\Omega)} \\
& \leq  c_\Omega \Vert \eta_h(\mH u-\mH_h u)\Vert_{H^1_0(\Omega)} \Vert g\Vert_{L^2(\Omega)} 
   \leq c_\Omega c_{\mathcal{A}} \Vert \mA(\mH u-\mH_h u)\Vert_{H^{-1}(\Omega)} \Vert g\Vert_{L^2(\Omega)} \\
  & \leq
  c_\Omega  c_{\mathcal{A}} M_{\mA} \Vert \mH u-\mH_h u\Vert_{H^1(\Omega)} \Vert g\Vert_{L^2(\Omega)} \\
  & \leq c_\Omega  c_{\mathcal{A}} M_{\mA}(c_i C_0+C_1) h^s \Vert  u\Vert_{W^{3/2,2}_{\vec\beta}(\Gamma)} \Vert g\Vert_{L^2(\Omega)}.
\end{align*}
The estimate for III follows directly from the estimate \eqref{eq::err:L2Harm} in Lemma \ref{L5.9} with $\beta_j>1-\lambda_{\Delta,j}$. Therefore, the result follows with $c_{\eqref{eq:.sterrest}} = c_i  c_{\eqref{eq::erretaL2_bis}} M_{\mH} + c_\Omega  c_{\mathcal{A}} M_{\mA}(c_i C_0+C_1) + c_i C_0$.   
\end{proof}

\begin{remark}\label{re::R5.19}We conjecture that the order of convergence in Lemma \ref{le::stateerror} can be doubled to obtain an approximation of order $h^{2s}$. Nevertheless, we could not profit from this improvement since the final order of convergence will be limited by the approximation rate of the other quantities, like the interpolation error estimate, cf. the proof of Lemma \ref{le::L6.5},  the error estimate obtained in Lemma \ref{le::L6.9} for the approximation of the Steklov-Poicar\'{e} operator, or the error estimate for the approximation of the adjoint operator obtained in the proof of Lemma \ref{le::L6.10}.
\end{remark}

\section{Numerical approximation. The control problem}\label{S6}
We define a discrete Dirichlet-to-Neumann or Poincare-Steklov operator $\mD_h:H^{1/2}(\Gamma)\to H^{-1/2}(\Gamma)$ via 
\[\langle \mD_h u, v\rangle_\Gamma :=  (\nabla \mH_h u,\nabla\mE_h v)_\Omega\ \forall v\in H^{1/2}(\Gamma),\]
where $\mE_h$ is a discrete extension operator.
\begin{remark}
The definition is independent of $\mE_h$: if $\mE_{h,1}$ and $\mE_{h,2}$ are discrete extension operators, then $(\mE_{h,1}-\mE_{h,2}) v\in Y_{0,h}$ for all $v\in H^{1/2}(\Gamma)$ and therefore $(\nabla \mH_h u,\nabla(\mE_{h,1}-\mE_{h,2})v)_\Omega=0$. 
\end{remark}
\begin{remark}\label{re::R6.2}
    This discretization of the Steklov-Poincar\'{e} operator $\mD$ is essentially different from the one used in \cite[eq. (3.15)]{OfPhanSteinbach2015}, where a {\em continuous} extension operator is used in the theory. On the other hand, in \cite[eqs. (5.6), (5.7)]{GongMateosSinglerZhang2022} a discrete extension operator based on the $L^2(\Gamma)$-projection is proposed. See Remark \ref{re::R6.7} about the different proof techniques that these choices imply and Remark \ref{re::R7.1} about the consequences in the practical implementation of the discretization. 
    
    The discrete normal derivative $\partial_n^h y_h$ used in \cite{Winkler_NM_2020} is built using the $L^2(\Gamma)$-projection. In our case, it also holds that for every $ u\in H^{1/2}(\Gamma)$, the linear functional $\mD_h u$ can be represented by the unique discrete function $\partial_n^h \mH_h  u\in U_h$ that satisfies 
    \[(\partial_n^h \mH_h  u, \mQ_h v)_{L^2(\Gamma)} = \langle \mD_h u, v\rangle_\Gamma\ \forall v\in H^{1/2}(\Gamma).\]%
    Nevertheless, we will not use this function neither for theoretical nor for practical purposes. 
\end{remark}

\medskip
We have that for every $v\in H^{1/2}(\Gamma)$,  $\langle \mD_h u, v\rangle_\Gamma = (\nabla \mH_h u,\nabla\mH_h v)_\Omega$, and hence the bilinear form $\langle \mD_h u, v\rangle_\Gamma$ is symmetric. Moreover, 
\[\langle \mD_h u_h, u_h\rangle_\Gamma^{1/2} = (\nabla \mH_h u_h,\nabla\mH_h u_h)_\Omega^{1/2}\]
is a seminorm in $U_h$. We show below in Remark \ref{re::R6.5} that it is a seminorm equivalent to $\Vert\mH u_h\Vert_{L^2(\Omega)}$, and hence to the Slobodetskii seminorm $\vert u_h\vert_{H^{1/2}(\Gamma)}$, with equivalence constants independent of $h$. Using Lemma \ref{Le::contHh} and Remark \ref{re::R5.10}, we have that the operator norm of $\mD_h$ is bounded by the constant $M_{\mH}^2$. Doing the same kind of abuse of notation as in the mentioned remark, we will denote this $M_{\mD}$.

In the rest of this section we will assume $0<h<h_\mA$. We define $\functional_h:H^{1/2}(\Gamma)\to \mathbb{R}$ as
\[\functional_h(u) = \frac12\Vert y_h(u)-y_d\Vert_{L^2(\Omega)}^2 +
\frac{\tikhonov}{2}\langle \mD_h (u-u_d),u-u_d\rangle_\Gamma\]
and consider the discrete problem
\[\Pbh\qquad \min_{u_h\in U_h} \functional_h(u_h).\]
We obtain now different expressions for $\functional_h'(u)$ that will help us to describe a practical optimization procedure, cf. \eqref{eq::discretederivative}, or to obtain error estimates, cf. \eqref{OCh}.
For every $u,v\in H^{1/2}(\Gamma)$
\[\functional_h'(u)v = (y_h(u)-y_d,y_h(v))_\Omega + \kappa \langle \mD_h (u-u_d), v\rangle_\Gamma.\]
For every $g\in H^{-1}(\Omega)$, we define $\phi_h(g)\in Y_{0,h}$ the unique solution of
\begin{equation}\label{eq::phih}
    \mathfrak{a}(\zeta_h,\phi_h(g)) = \langle g,\zeta_h\rangle_\Omega\ \forall \zeta_h\in Y_{0,h}.
\end{equation}
Since this is a linear system whose coefficient matrix is the transpose of that of the problem studied in Lemma \ref{L5.1}, existence and uniqueness of solution of this equation follows directly from Lemma \ref{L5.1}, as well as the estimate
\begin{equation}\label{E6.1}\Vert\phi_h(g)\Vert_{H^1_0(\Omega)} \leq c_{\mathcal{A}}\Vert g\Vert_{H^{-1}(\Omega)}.\end{equation}
With the same technique as for \eqref{eq::erreta}, we have that for $g\in L^2_{\vec\beta}$, where $\vec\beta$ satisfies $\beta_j\geq 0$, $1-\lambda_{A,j} < \beta_j < 1$, $s\leq 1$ such that $s \leq (1-\beta_j)/\mu_j$ for all $j\in\{1,\ldots,m\}$, we have
\begin{equation}\label{eq::errphi}
    \Vert\phi_g-\phi_h(g)\Vert_{H^1_0(\Omega)} \leq c_{\eqref{eq::erreta}} h^s \Vert g\Vert_{L^2_{\vec\beta}(\Omega)}.
\end{equation}
Let $\varphi_h(u) = \phi_h(y_h(u)-y_d)$. We have that
\begin{align*}
  (y_h(u)-y_d,y_h(v))_\Omega   &= (y_h(u)-y_d,\eta_h(\mE_h  v))_\Omega + (y_h(u)-y_d,\mE_h v)_\Omega \\ &=\mathfrak{a}(\eta_h(\mE_h v),\varphi_h(u)) + (y_h(u)-y_d,\mE_h v)_\Omega\\
   &= -\mathfrak{a}(\mE_h v,\varphi_h(u)) + (y_h(u)-y_d,\mE_h v)_\Omega,
\end{align*}
and hence
\begin{equation}\label{eq::discretederivative}
\functional_h'(u) v = -\mathfrak{a}(\mE_h v,\varphi_h(u)) + (y_h(u)-y_d,\mE_h v)_\Omega + \tikhonov (\nabla \mH_h (u-u_d),\nabla\mE_h v)_\Omega.
\end{equation}
We can also write the following. The adjoint operator of $\mS_h$ is $\mS_h^\star:\dualH \to H^{-1/2}(\Gamma)$. For every $g\in \dualH$, $\mS_h^\star g\in H^{-1/2}(\Gamma)$ and satisfies that for all $v\in H^{1/2}(\Gamma)$ and every discrete extension operator $\mE_h$,
\begin{align}
  \langle \mS_h^\star g,v\rangle_\Gamma &= \langle g, \mS_h v\rangle_\Omega = \langle g, \eta_h(\mE_h v)\rangle_\Omega +
  \langle g, \mE_h v\rangle_\Omega \notag \\
  &= \mathfrak{a}(\eta_h(\mE_h v),\phi_h(g)) + \langle g, \mE_h v\rangle_\Omega  = -\mathfrak{a}(\mE_h v,\phi_h(g)) + \langle g, \mE_h v\rangle_\Omega.
  \label{eq::SHstar}
\end{align}
A straight forward consequence of Theorem \ref{th::T5.17} and Remark \ref{re::R5.18} is that the operator norm of $\mS^\star_h$ is bounded independently of $h$ by the constant $M_{\mS}>0$.
Using $g=y_h(u)-y_d$ in \eqref{eq::SHstar} and remembering that $\varphi_h(u)=\phi_h(y_h(u)-y_d)$, and $y_h(u)=S_hu$,  the derivative can be expressed as 
\[\functional_h'(u) v = \langle \mS_h^\star \mS_h u + \tikhonov \mD_h u,v\rangle_\Gamma - \langle  \mS_h^\star y_d + \tikhonov \mD_h u_d,v\rangle_\Gamma\ \forall u,v\in H^{1/2}(\Gamma).
\]
Remembering that before Lemma \ref{L3.2} we denoted $\mT = \mS^\star\mS+\tikhonov \mD$ and $w = \mS^\star y_d + \tikhonov \mD u_d$, we now introduce
\[
\mT_h = \mS^\star_h\mS_h+\tikhonov \mD_h\text{ and }w_h=\mS_h^\star y_d + \tikhonov \mD_h u_d
\]
such that 
\begin{align}
    \functional_h'(u)v=\langle\mT_hu-w_h,v\rangle_\Gamma. \label{OCh}
\end{align}
We remark again that the operator norm of $\mT_h$ is bounded independently of $h$ by a constant that we also denote $M_{\mT}$, again abusing notation.

The proof of the following key result is similar to that of Lemma \ref{L3.3}, but there are some important details that are different, so we write it in detail. Notice that we have to use  the convergence results for finite element approximations.

\begin{theorem}\label{L6.1}
There exists $\nu^\star>0$ independent of $h$ such that \[\functional_h''(u) v_h^2\geq \nu^\star\Vert v_h\Vert_{H^{1/2}(\Gamma)}^2\ \forall v_h\in U_h.\]
\end{theorem}
\begin{proof}
  We have that
  \begin{align*}
  \functional_h''(u) v_h^2 &= (y_h(v_h),y_h(v_h))_\Omega + \kappa (\mD_h v_h,  v_h)_\Gamma \\
  & = \Vert\eta_h(\mH_h v_h) + \mH_h v_h\Vert_{L^2(\Omega)}^2 + \tikhonov \Vert\nabla \mH_h v_h\Vert_{L^2(\Omega)}^2.
  \end{align*}
If the statement is false, there exist sequences $\{h_n\}_n$ and  $\{v_{h_n}\}_n\subset H^{1/2}(\Gamma)$ such that for all $n\geq 1$, it holds that $v_{h_n}\in U_{h_n}$, $\Vert v_{h_n}\Vert_{H^{1/2}(\Gamma)} = 1$ and
\begin{equation}\label{eq::contr_assumpt}\Vert\eta_{h_n}(\mH_{h_n}v_{h_n}) + \mH_{h_n} v_{h_n}\Vert_{L^2(\Omega)}^2 + \tikhonov \Vert\nabla \mH_{h_n} v_{h_n}\Vert_{L^2(\Omega)}^2\leq \frac{1}{n}.\end{equation}
Since $\{v_{h_n}\}_n$ is bounded in $H^{1/2}(\Gamma)$, 
there exist a subsequence $\{v_{h_n}\}_n$, that we do not relabel, and $v\in H^{1/2}(\Gamma)$, such that $v_{h_n}\rightharpoonup v$ weakly in $H^{1/2}(\Gamma)$ as $n\to\infty$.  

We first deduce that $\mH_{h_n} v_{h_n} \rightharpoonup \mH v $ weakly in $H^1(\Omega)$ as $n\to\infty$ as follows: for any  $g\in\dualH$ we have that
\begin{align*}
  \langle \mH_{h_n} v_{h_n} - \mH v,g\rangle_\Omega &= 
  \langle \mH_{h_n} v_{h_n} -\mH_{h_n} v,g\rangle_\Omega  + \langle \mH_{h_n} v- \mH v,g\rangle_\Omega \\
   &= \langle \mH_{h_n}^\star g, v_{h_n}-v\rangle_\Gamma + \langle \mH_{h_n} v- \mH v,g\rangle_\Omega.
\end{align*}
The second term tends to zero due to Lemma \ref{co::conv::discharm}. For the first one we have
\begin{align*}
  \langle \mH_{h_n}^\star g, v_{h_n}-v\rangle_\Gamma  &= 
  \langle \mH_{h_n}^\star g-\mH^\star g , v_{h_n}-v\rangle_\Gamma +
  \langle \mH^\star g, v_{h_n}-v\rangle_\Gamma.
\end{align*}
The second addend tends to zero due to the weak convergence $v_{h_n}\rightharpoonup v$. For the first one we use that $v_{h_n}-v$ is bounded in $H^{1/2}(\Gamma)$ and the strong convergence $\mH_{h_n}^\star g\to\mH^\star g$ in $H^{-1/2}(\Gamma)$, which follows from Lemma \ref{MML5x}.

We show next that $\eta_{h_n}(\mH_{h_n}v_{h_n})\rightharpoonup \eta_{\mH v}$ weakly in $H^1_0(\Omega)$ as $n\to\infty$.
From the error estimate \eqref{eq::erreta}, the stability result \eqref{eq::boundHhuhL2}, and the fact that $\Vert v_{h_n}\Vert_{H^{1/2}(\Gamma)} = 1$, we deduce that \[\Vert \eta_h(\mH_{h_n}v_{h_n})-\eta_{_{\mH_{h_n}v_{h_n}}}
\Vert_{H^1_0(\Omega)} \to 0\text{ as }n\to\infty.\] On the other hand, for every $g\in H^{-1}(\Omega)$, 
\begin{align*}
\langle g, \eta_{_{\mH_{h_n}v_{h_n} - \mH v}}\rangle_\Omega & = \mathfrak{a}(\eta_{_{\mH_{h_n}v_{h_n} - \mH v}},\phi_g) \\
&= -\langle \mA (  \mH_{h_n}v_{h_n} - \mH v),\phi_g \rangle_\Omega  = -\langle   \mH_{h_n}v_{h_n} - \mH v, \mA^\star\phi_g \rangle_\Omega
\end{align*}
which tends to zero since we have already proved that $\mH_{h_n}v_{h_n}\rightharpoonup \mH v$ weakly in $H^1(\Omega)$ as $n\to\infty$. Therefore $\eta_{h_n}(\mH_{h_n}v_{h_n})\rightharpoonup \eta_{\mH v}$ weakly in $H^1(\Omega)$ as $n\to\infty$.

Using that $H^1(\Omega)$ is compactly embedded in $L^2(\Omega)$, we obtain that $\eta_{h_n}(\mH_{h_n}v_{h_n}) + \mH_{h_n} v_{h_n}\to \eta_{_{\mH v}} + \mH v$ strongly in $L^2(\Omega)$, and by \eqref{eq::contr_assumpt}, $\eta_{_{\mH v}} + \mH v =0$. This, together with the fact that $\eta_{_{\mH v}}\in H^1_0(\Omega)$, implies that $v=0$. So $\mH_{h_n} v_{h_n}\to 0$ in $L^2(\Omega)$. Since we have by assumption \eqref{eq::contr_assumpt} that $\nabla \mH_{h_n} v_{h_n}\to 0$ in $L^2(\Omega)$, then $\Vert\mH_{h_n} v_{h_n}\Vert_{H^1(\Omega)}\to 0$. But, using \eqref{E3.1}
and the fact that  $(\mH_{h_n} v_{h_n})_{\vert\Gamma} \equiv  v_{h_n}$ on $\Gamma$, we have that
\[1 = \Vert v_{h_n}\Vert_{H^{1/2}(\Gamma)} \leq M_{\tr} \Vert\mH_{h_n} v_{h_n}\Vert_{H^1(\Omega)}\to 0\]
and we obtain a contradiction.
\end{proof}
\begin{remark}\label{re::R6.4}In \cite[Theorem 3.2]{OfPhanSteinbach2015}, the conclusion of Theorem \ref{L6.1} is formulated as an assumption.
In \cite[Remark 2.5]{PalGudi2026} and \cite[Lemma 3.9]{OfPhanSteinbach2015} similar statements to that of Theorem \ref{L6.1} are done, but with constants depending on $h$. 
In the proofs of \cite[Theorem 5.7]{GongMateosSinglerZhang2022} and \cite[Lemma 12]{Winkler_NM_2020}, this result is used without a complete justification.
\end{remark}

\begin{remark}\label{re::R6.5}
A consequence of Lemma \ref{Le::contHh} and Theorem \ref{L6.1} applied to the case $\mA = -\Delta$, is that $\Vert \mH u_h\Vert_{H^1(\Omega)}$ and $\Vert\mH_h u_h\Vert_{H^1(\Omega)}$ are equivalent norms in $U_h$, with equivalence constants independent of $h$. Following the same argument as in the proof of Lemma \ref{le::Lemma3.1} leads to the equivalence of the seminorms  $\Vert \nabla \mH u_h\Vert_{L^2(\Omega)}$ and $\Vert\nabla\mH_h u_h\Vert_{L^2(\Omega)}$.
\end{remark}

Gathering the above considerations, the following result is immediate.
\begin{corollary}Assume $y_d\in L^2(\Omega)$ and $u_d\in H^{1/2}(\Gamma)$. Then, there exists a $h_{\mA}>0$ such that problem \Pbh has a unique solution $\bar u_h\in U_h$. In addition, there exist unique functions $\bar y_h\in Y_h$ and $\bar\varphi_h\in Y_{0,h}$ such that, for every extension operator $\mE_h$, we have
\begin{subequations}
\begin{align}
\mathfrak{a}(\bar y_h,\zeta_h) &= 0\ \forall \zeta_h\in Y_{0,h},\ \bar y_h\equiv \bar u_h\text{ on }\Gamma,\\
\mathfrak{a}(\zeta_h,\bar \varphi_h) &= (\bar y_h-y_d,\zeta_h)_{\Omega}\ \forall \zeta_h\in Y_{0,h},\\
 -\mathfrak{a}(\mE_h v_h,\bar\varphi_h) & + (\bar y_h,\mE_h v_h)_\Omega + \tikhonov (\nabla \mH_h \bar u_h,\nabla\mE_h v_h)_\Omega \notag \\
 & =
 (y_d,\mE_h v_h)_\Omega +\tikhonov (\nabla \mH_h u_d,\nabla\mE_h v_h)_\Omega\ \forall v_h\in U_h.
\end{align}
\end{subequations} 
Moreover, we can write
\begin{equation}\label{eq::fooch}
    \langle (\mS_h^\star\mS_h + \tikhonov\mD_h)\bar u_h,v_h\rangle = \langle w_h,v_h\rangle \ \forall v_h\in U_h.
\end{equation}
  
\end{corollary}

Let $\bar u\in H^{1/2}(\Gamma)$ and $\bar u_h$ be respectively the solutions of \Pb and $\Pbh$. Due to \eqref{OC} and \eqref{eq::fooch} the first order optimality condition for the continuous and the discrete problem read as:
\begin{align}\label{eq::firstorder}\langle \mT \bar u, v \rangle_\Gamma &= \langle w ,v\rangle_\Gamma\  \forall v\in H^{1/2}(\Gamma).\\
\label{eq::firstorderh}\langle \mT_h \bar u_h, v_h \rangle_\Gamma &= \langle w_h ,v_h\rangle_\Gamma\ \forall v_h\in U_h.\end{align}
We introduce now an intermediate discrete control. Define the finite element solution of \eqref{eq::firstorder} as the unique $u_h^\star\in U_h$ such that
\begin{equation}\label{eq::uhstar}\langle \mT   u_h^\star, v_h \rangle_\Gamma = \langle w ,v_h\rangle_\Gamma\ \forall v_h\in U_h.\end{equation}

\begin{lemma}The estimate
\begin{equation}\label{eq::estimateustar}
\Vert\bar u-u_h^\star\Vert_{H^{1/2}(\Gamma)} \leq  \frac{M_{\mT} M_{\tr} c_{\mI}}{\nu} h^s\Vert\bar u\Vert_{W^{3/2,2}_{\vec\beta}(\Gamma)}
\end{equation}
holds, where $\vec\beta$ satisfies $\beta_j\geq 0$ and $1-\lambda_j < \beta_j<1$ for all $j\in\{1,\ldots,m\}$ and $s\leq 1$ satisfies $s \leq (1-\beta_j)/\mu_j$ for all $j\in\{1,\ldots,m\}$.
\end{lemma}

\begin{proof}
Thanks to Lemma \ref{L3.3}, we can apply the C\'{e}a Lemma. Taking into account the regularity of $\bar u$ we can use the interpolation error estimate \eqref{eq:interpolation_error_Gamma} to obtain
\[
\Vert\bar u-u_h^\star\Vert_{H^{1/2}(\Gamma)} \leq \frac{M_{\mT}}{\nu} \inf_{v_h\in U_h} \Vert\bar u-v_h\Vert_{H^{1/2}(\Gamma)} \leq \frac{M_{\mT} M_{\tr} c_{\mI}}{\nu} h^s\Vert\bar u\Vert_{W^{3/2,2}_{\vec\beta}(\Gamma)}
\]
and the result follows.
\end{proof}

\begin{lemma}\label{le::L6.6}
    Let $\nu^\star>0$ be the constant independent of $h$ found in Theorem \ref{L6.1}. Then it holds
    \begin{align}
    \label{eq::generalestimate}
    \nu^\star\Vert u_h^\star-\bar u_h\Vert_{H^{1/2}(\Gamma)}^2 \leq &
    \langle (\mT_h-\mT)\bar u, u_h^\star-\bar u_h\rangle_\Gamma+
    \langle w-w_h, u_h^\star-\bar u_h\rangle_\Gamma\\
    & +
    \langle (\mT_h-\mT)(u_h^\star-\bar u), u_h^\star-\bar u_h\rangle_\Gamma.\notag
    \end{align}
\end{lemma}

\begin{proof}
    Using Theorem \ref{L6.1} and noticing that $\functional_h''(u)v_h^2 = \langle \mT_h v_h,v_h\rangle_\Gamma$ for all $v_h\in U_h$ we have that for some $c>0$
    \begin{align*}
        \nu^\star\Vert u_h^\star-\bar u_h\Vert_{H^{1/2}(\Gamma)}^2  &\leq
        \langle \mT_h(u_h^\star-\bar u_h), u_h^\star-\bar u_h\rangle_\Gamma\\
        &=\langle (\mT_h-\mT)u_h^\star, u_h^\star-\bar u_h\rangle_\Gamma+
        \langle w-w_h,u_h^\star-\bar u_h\rangle_\Gamma \\
        &=\langle (\mT_h-\mT)\bar u , u_h^\star-\bar u_h\rangle_\Gamma+
        \langle w-w_h,u_h^\star-\bar u_h\rangle_\Gamma\\ &\qquad+
        \langle (\mT_h-\mT)(u_h^\star-\bar u) , u_h^\star-\bar u_h\rangle_\Gamma,
    \end{align*}
where in the first equality we have used both $\eqref{eq::firstorderh}$ and \eqref{eq::uhstar}.
\end{proof}

\medskip
Next we estimate each of the terms of the right hand side of \eqref{eq::generalestimate}, starting with the third one.

\begin{lemma}\label{le::L6.5}For every $v_h\in U_h$, the following estimate holds:
    \[\langle (\mT_h-\mT)(u_h^\star-\bar u) , v_h\rangle_\Gamma\leq 2 M_{\mT} M_{\tr} c_{\mI} h^s  \Vert \bar u\Vert _{W^{3/2,2}_{\vec\beta}(\Gamma)} \Vert v_h\Vert_{H^{1/2}(\Gamma)}
    \]
     where  $\vec\beta$ satisfies $\beta_j\geq 0$ and $1-\lambda_j < \beta_j<1$ for all $j\in\{1,\ldots,m\}$ and $s\leq 1$ satisfies $s \leq (1-\beta_j)/\mu_j$ for all $j\in\{1,\ldots,m\}$.
\end{lemma}
\begin{proof}
Using the continuity of $\mT$ and of $\mT_h$ and \eqref{eq::estimateustar}
    \begin{align*}
        \langle (\mT_h-\mT)(u_h^\star-\bar u) , v_h\rangle_\Gamma & \leq
        \Vert (\mT_h-\mT)(u_h^\star-\bar u)\Vert_{H^{-1/2}(\Gamma)} \Vert v_h\Vert_{H^{1/2}(\Gamma)} \\
        & \leq 2 M_{\mT} \Vert u_h^\star-\bar u\Vert_{H^{1/2}(\Gamma)} \Vert v_h\Vert_{H^{1/2}(\Gamma)} \\
        & \leq 2 M_{\mT} M_{\tr} c_{\mI}  h^s \Vert \bar u\Vert _{W^{3/2,2}_{\vec\beta}(\Gamma)} \Vert v_h\Vert_{H^{1/2}(\Gamma)}
    \end{align*}
    and the result follows.
\end{proof}

\medskip
We split the first term of the right hand side of \eqref{eq::generalestimate}, $\langle (\mT_h-\mT)\bar u, v_h\rangle_\Gamma = \langle (\mD_h-\mD)\bar u, v_h\rangle_\Gamma + \langle (\mS_h^\star\mS_h-\mS^\star\mS)\bar u, v_h\rangle_\Gamma$, and estimate each of the two terms separately. 
\begin{remark}\label{re::R6.7}
  In the proof of the next two lemmas, we find one of the main differences with the proofs done in \cite[Theorem 1]{Winkler_NM_2020} or in \cite[Theorem 5.7]{GongMateosSinglerZhang2022}. We cannot write $(\mD-\mD_h)\bar u = \mD\bar u - \mQ_h\mD\bar u + \mQ_h\mD\bar u -\mD_h\bar u$ or $(\mS^\star\mS - \mS_h^\star \mS_h)\bar u = \mS^\star\mS \bar u - \mQ_h \mS^\star\mS \bar u +  \mQ_h \mS^\star\mS \bar u -\mS_h^\star \mS_h \bar u$ because $\mD\bar u$ and $\mS^\star\mS \bar u$ belong to $\Wsj$, but not necessarily to $H^{1/2}(\Gamma)$, so $\mQ_h\mD\bar u$ or $\mQ_h \mS^\star\mS \bar u$ may not be well defined.
\end{remark}

\begin{lemma}\label{le::L6.9}Let $C_1>0$ be the constant found in Corollary \ref{C5.8}. Then, 
\[\langle (\mD-\mD_h)\bar u,v_h\rangle_\Gamma \leq C_1 M_{\mH} h^s \Vert \bar u\Vert _{W^{3/2,2}_{\vec\beta}(\Gamma)} \Vert v_h\Vert_{H^{1/2}(\Gamma)}\ \forall v_h\in U_h,\]
 where $\vec\beta$ satisfies $\beta_j\geq 0$ and $1-\lambda_j < \beta_j<1$ for all $j\in\{1,\ldots,m\}$ and $s\leq 1$ satisfies $s \leq (1-\beta_j)/\mu_j$ for all $j\in\{1,\ldots,m\}$.
\end{lemma}
\begin{proof}
    Using the definition of $\mD$ and $\mD_h$, the fact that both $\mH v_h\equiv v_h$ on $\Gamma$ and $\mH_h v_h\equiv v_h$ on $\Gamma$, which implies that $\mH v_h-\mH_h v_h\in H^1_0(\Omega)$, the definition of harmonic extension, Corollary \ref{C5.8} and Lemma \ref{eq::stabHh1} together with Remark \ref{re::R5.10}, we obtain
    \begin{align*}
        \langle(\mD-\mD_h)\bar u, v_h\rangle_\Gamma & = (\nabla \mH\bar u, \nabla\mH v_h)_\Omega - (\nabla \mH_h\bar u, \nabla\mH_h v_h)_\Omega \\
        & = (\nabla \mH\bar u, \nabla(\mH-\mH_h) v_h)_\Omega + (\nabla \mH\bar u -\nabla \mH_h\bar u, \nabla\mH_h v_h)_\Omega \\
        & =  (\nabla \mH\bar u -\nabla \mH_h\bar u, \nabla\mH_h v_h)_\Omega 
        \leq C_1 M_{\mH} h^s \Vert \bar u\Vert_{W^{3/2,2}_{\vec\beta}(\Gamma)} \Vert v_h\Vert_{H^{1/2}(\Gamma)}.
    \end{align*}
Notice that the relations among $\beta_j$, $\lambda_j$ and $s$ together with the inequality $\lambda_j\leq \lambda_{\Delta,j}$ imply that we can use Corollary \ref{C5.8}.
\end{proof}
\begin{lemma}\label{le::L6.10}
There exists a constant $c_{\eqref{eq::SSerr}}>0$, that may depend on $\vec\mu$ but is independent of $h$, such that for every $v_h\in U_h$
\begin{equation}\label{eq::SSerr}
\langle (\mS^\star\mS - \mS_h^\star \mS_h)\bar u, v_h\rangle_\Gamma \leq c_{\eqref{eq::SSerr}} h^s \Vert \bar u\Vert _{W^{3/2,2}_{\vec\beta}(\Gamma)} \Vert v_h\Vert_{H^{1/2}(\Gamma)}
\end{equation}
 where $\vec\beta$ satisfies $\beta_j\geq 0$ and $1-\lambda_j < \beta_j<1$ for all $j\in\{1,\ldots,m\}$ and $s\leq 1$ satisfies $s \leq (1-\beta_j)/\mu_j$ for all $j\in\{1,\ldots,m\}$.
\end{lemma}
\begin{proof}
We write the expression as
\begin{align*}
\langle (\mS^\star\mS - \mS_h^\star \mS_h)\bar u, v_h\rangle_\Gamma & = 
\langle (\mS^\star\mS - \mS_h^\star \mS)\bar u, v_h\rangle_\Gamma +
\langle (\mS_h^\star\mS - \mS_h^\star \mS_h)\bar u, v_h\rangle_\Gamma \\
 & = \langle (\mS^\star - \mS_h^\star)\bar y, v_h\rangle_\Gamma +
\langle (\mS -  \mS_h)\bar u, \mS_h v_h\rangle_\Omega = I + II
\end{align*} 
We first estimate I. Using the expressions for $\mS$ and $\mS^\star$ obtained in \eqref{MME3.3} and \eqref{eq::SHstar}, the definitions of $\phi$ and $\phi_h$ provided in \eqref{eq::Sstar} and \eqref{eq::phih} and the fact that
$\mH v_h-\mH_h v_h\in H^1_0(\Omega)$
\begin{align*}
    I & = \langle (\mS^\star - \mS_h^\star)\bar y, v_h\rangle_\Gamma \\
    & = -\mathfrak{a}(\mH v_h,\phi_{\bar y}) + \mathfrak{a}(\mH_h v_h,\phi_h(\bar y)) +(\bar y,\mH v_h-\mH_h v_h)_\Omega \\
    & = -\mathfrak{a}(\mH v_h-\mH_h v_h,\phi_{\bar y}) + \mathfrak{a}(\mH_h v_h,\phi_h(\bar y)-\phi_{\bar y}) +(\bar y,\mH v_h-\mH_h v_h)_\Omega \\
    & = \mathfrak{a}(\mH_h v_h,\phi_h(\bar y)-\phi_{\bar y}) +(\bar y,\mH v_h-\mH_h v_h)_\Omega = A+B
\end{align*}
For A, we use the continuity of the bilinear form $\mathfrak{a}(\cdot,\cdot)$, Lemma \ref{Le::contHh} together with Remark \ref{re::R5.10}, the finite element error estimate \eqref{eq::errphi}, the injection $H^1(\Omega)\hookrightarrow L^2_{\vec\beta}(\Omega)$, Theorem \ref{T2.2}, and the injection $W^{3/2,2}_{\vec\beta}(\Gamma)\hookrightarrow H^{1/2}(\Gamma)$
\begin{align*}
    A  &= \mathfrak{a}(\mH_h v_h,\phi_h(\bar y)-\phi_{\bar y}) 
    \leq M_{\mA}\Vert \mH_h v_h\Vert_{H^1(\Omega)} \Vert \phi_h(\bar y)-\phi_{\bar y}\Vert_{H^1(\Omega)} \\
    &\leq  c_{\eqref{eq::erreta}} M_{\mA} M_{\mH} \Vert v_h\Vert_{H^{1/2}(\Gamma)} h^s \Vert \bar y\Vert_{L^2_{\vec\beta}(\Omega)} \\
    &\leq c_i^2 c_{\eqref{eq::erreta}} M_{\mA} M_{\mH} M_{\mS} h^s  \Vert \bar u\Vert_{W^{3/2,2}_{\vec \beta}(\Gamma)} \Vert v_h\Vert_{H^{1/2}(\Gamma)}.
\end{align*}
To estimate B we use \eqref{eq::err:L2Harm} to obtain
\begin{align*}
B &= (\bar y,\mH v_h-\mH_h v_h)_\Omega \leq \Vert \bar y\Vert_{L^2(\Omega)} \Vert \mH v_h-\mH_h v_h\Vert_{L^2(\Omega)} \\ &\leq M_{\mS} c_i^2 C_0 h^s \Vert \bar u\Vert_{W^{3/2,2}_{\vec \beta}(\Gamma)}  \Vert v_h\Vert_{H^{1/2}(\Gamma)}
\end{align*}
Finally, using Lemma \ref{le::stateerror} we obtain directly
\begin{align*}
    II &\leq \Vert \bar y -\bar y_h\Vert_{L^2(\Omega)} \Vert S_h v_h\Vert_{L^2(\Omega)}\\
    & \leq c_{\eqref{eq:.sterrest}} h^s \Vert \bar u \Vert_{W^{3/2,2}_{\vec\beta}(\Gamma)}
    \Vert S_h v_h\Vert_{L^2(\Omega)}  \leq c_{\eqref{eq:.sterrest}} M_{\mS} h^s \Vert \bar u\Vert_{W^{3/2,2}_{\vec \beta}(\Gamma)}  \Vert v_h\Vert_{H^{1/2}(\Gamma)},
\end{align*}
and the proof is complete.
\end{proof}

\medskip
Finally we treat the second term of the right hand side of \eqref{eq::generalestimate}.
\begin{lemma}\label{le::L6.11}
    Assume $y_d\in L^2(\Omega)$ and $u_d\in W^{3/2,2}_{\vec\beta}(\Gamma)$, where $\vec\beta$ satisfies $\beta_j\geq 0$ and $1-\lambda_j < \beta_j<1$ for all $j\in\{1,\ldots,m\}$. Then, there exists $C>0$, that may depend on $\vec\mu$ but is independent of $h$, such that
    \[\langle w- w_h,v_h\rangle_\Gamma \leq C h^s (\Vert y_d\Vert_{L^2(\Omega)} + \Vert u_d\Vert_{W^{3/2,2}_{\vec\beta}(\Gamma)})\Vert v_h\Vert_{H^{1/2}(\Gamma)}\ \forall v_h\in H^{1/2}(\Gamma),\]
    where $s\leq 1$ satisfies $s \leq (1-\beta_j)/\mu_j$ for all $j\in\{1,\ldots,m\}$.
\end{lemma}
\begin{proof}
    Noting that $w-w_h  = (\mS^\star-\mS_h^\star)y_d + \tikhonov(\mD-\mD_h)u_d$, the proof follows the same lines as that of Lemma \ref{le::L6.9} and the derivation of the estimate for the term I in the proof of Lemma \ref{le::L6.10}, noting that to estimate the term A we use that $y_d\in L^2(\Omega)\hookrightarrow L^2_{\vec\beta}(\Omega)$ and to estimate the term B we use that $y_d\in L^2(\Omega)$.
\end{proof}

\begin{theorem}\label{th::mainTheorem}Assume $y_d\in L^2(\Omega)$ and $u_d\in W^{3/2,2}_{\vec\beta}(\Gamma)$, where $\vec\beta$ satisfies $\beta_j\geq 0$ and $1-\lambda_j < \beta_j<1$ for all $j\in\{1,\ldots,m\}$. Let $\bar u$ and $\bar u_h$ be  the unique solutions of $\Pb$ and $\Pbh$, respectively. Then, there exist $h_{\mA}>0$ and $C>0$, that may depend on $\vec\mu$ but is independent of $h$, such that, for all $0<h<h_{\mA}$,
    \[\Vert \bar y -\bar y_h\Vert_{H^1(\Omega)} + \Vert \bar u -\bar u_h\Vert_{H^{1/2}(\Gamma)} + \Vert\bar\varphi-\bar\varphi_h\Vert_{H^1_0(\Omega)} \leq C h^s (\Vert y_d\Vert_{L^2(\Omega)} + \Vert u_d\Vert_{W^{3/2,2}_{\vec\beta}(\Gamma)}) \]
    where $s\leq 1$ satisfies $s \leq (1-\beta_j)/\mu_j$ for all $j\in\{1,\ldots,m\}$.
\end{theorem}
\begin{proof}
    By the triangle inequality
    \[\Vert \bar u -\bar u_h\Vert_{H^{1/2}(\Gamma)} \leq \Vert \bar u - u^\star_h\Vert_{H^{1/2}(\Gamma)} + \Vert  u_h^\star -\bar u_h\Vert_{H^{1/2}(\Gamma)}.\]
    For the first term we apply estimate \eqref{eq::estimateustar}. To estimate the second one, we apply Lemma \ref{le::L6.6} to obtain a bound of $\Vert  u_h^\star -\bar u_h\Vert_{H^{1/2}(\Gamma)}^2$ in terms of the three addends in the right hand of \eqref{eq::generalestimate}. Each of these terms can be estimated using Lemmata \ref{le::L6.5}, \ref{le::L6.9}, \ref{le::L6.10} and \ref{le::L6.11} for the specific value $v_h= u_h^\star-\bar u_h$. After simplifying $\Vert  u_h^\star -\bar u_h\Vert_{H^{1/2}(\Gamma)}$ at both sides of the resulting inequality, we  obtain the existence of a constant $C>0$ independent of $h$ such that
    \[\Vert  u_h^\star -\bar u_h\Vert_{H^{1/2}(\Gamma)}\leq C h^s (\Vert \bar u\Vert_{W^{3/2,2}_{\vec\beta}(\Gamma)}+\Vert y_d\Vert_{L^2(\Omega)} + \Vert u_d\Vert_{W^{3/2,2}_{\vec\beta}(\Gamma)}).\]
    The estimate for the control follows from the stability of the solution with respect to the data, namely Theorem \ref{th::T4.14}.

    To obtain the estimate for the state variable, we apply Lemma \ref{le::stateerror}, estimate \eqref{E5.5} together with Remark \ref{re::R5.18}, the just obtained error estimate for the optimal control and Theorem \ref{th::T4.14}, we obtain 
    \begin{align*}
        \Vert \bar y -\bar y_h\Vert_{H^1(\Omega)}  & \le\Vert \bar y -y_h(\bar u)\Vert_{H^1(\Omega)} + \Vert  y_h(\bar u) -\bar y_h\Vert_{H^1(\Omega)} \\
        &\leq c_{\eqref{eq:.sterrest}} h^s \Vert \bar u\Vert_{W^{3/2,2}_{\vec\beta}(\Gamma)} + M_{\mS}\Vert \bar u -\bar u_h\Vert_{H^{1/2}(\Gamma)}\\ &\leq C h^s (\Vert y_d\Vert_{L^2(\Omega)} + \Vert u_d\Vert_{W^{3/2,2}_{\vec\beta}(\Gamma)})
    \end{align*}
 For the error estimate for the adjoint state variable we use the error estimate \eqref{eq::errphi} together with the embedding $L^2(\Omega)\hookrightarrow L^2_{\vec\beta}(\Omega)$ and the continuity estimate \eqref{E6.1} together with the embedding $\dualH\hookrightarrow H^1(\Omega)$ to obtain
    \begin{align*}
        \Vert\bar\varphi-\bar\varphi_h\Vert_{H^1_0(\Omega)} & = \Vert \phi_{\bar y -y_d} - \phi_h(\bar y_h-y_d)\Vert_{H^1_0(\Omega)} \\
        & \leq 
        \Vert \phi_{\bar y -y_d} - \phi_h(\bar y-y_d)\Vert_{H^1_0(\Omega)} + \Vert \phi_h(\bar y ) - \phi_h(\bar y_h\Vert_{H^1_0(\Omega)} \\
        & \leq c_i c_{\eqref{eq::erreta}} h^s \Vert \bar y - y_d\Vert_{L^2(\Omega)} + c_i c_{\mA} \Vert \bar y -\bar y_h\Vert_{H^1(\Omega)}
    \end{align*}
    and the estimate follows from the one just obtained for the state variable.
\end{proof}

\section{Numerical examples}\label{S7}
\subsection{Some computational details}
For a fixed triangulation $\mathcal{K}_h$ with nodes $\{x_k\}_{k=1}^N$, we consider the usual hat basis functions $\{\psi_k\}_{k=1}^N$ such that $\psi_i(x_k)= \delta_{ik}$. Functions in $Y_h$ can be written in the form $y_h = \sum_{j=1}^N y_j \psi_j$, and we will denote $\bm y = (y_1,\ldots,y_N)^\top$. We name $I$ and $B$ the sets of indexes of interior and boundary nodes respectively. Functions in $U_h$ can be written as $u_h = \sum_{k\in B}u_k\psi_{k\vert \Gamma}$, and in this case $\bm u = (u_1,\ldots,u_{N_\Gamma})^\top$. Let $M$ and $K$ be the mass and stiffness matrices, i.\,e., $M,K\in\mathbb{R}^{N\times N}$ with $m_{ij}=(\psi_j,\psi_i)_\Omega$ and $k_{ij}=(\nabla\psi_j,\nabla\psi_i)_\Omega$.

The restriction of $\mS_h$ (resp. $\mH_h$) to $U_h$ is a linear mapping between finite dimensional spaces. Consequently, there exists a matrix $S$ (resp. $H$) such that $y_h=\mS_h u_h$ if and only if $\bm y = S\bm u$ (resp. $z_h = \mH_h u_h$ if and only if $\bm z = H\bm u$). Therefore, we can write 
\[\functional_h(u_h) = \frac12 \bm u^\top T \bm u - \bm w^\top \bm u+ c,\]
where $T = S^\top M S + \tikhonov H^\top K H$ is a symmetric matrix, $\bm w$ is a vector and $c$ is a constant. The explicit computation of $T$ is out of the question, but, for given $\bm u$,  we can compute $T\bm u$ in an efficient way without computing $T$ itself.

Let $\bm A\in\mathbb{R}^{N\times N}$ with $a_{ij}=\mathfrak{a}(\psi_j,\psi_i)$ be the non-symmetric matrix related to the operator $\mathcal{A}$. Using the notational conventions $\bm A_{IB}$, $K_{B:}$, or $\bm y_I$, for example, to represent submatrices or subvectors, the state and adjoint state equations can be written as
\begin{align}
    \bm A_{II}\bm y_I &= -\bm A_{IB} \bm u, \label{eq::comp::st}\\
    \bm A^\top_{II}\bm \varphi_I &= M_{II}\bm y_I + M_{IB}\bm u - \tilde{\bm y}_I,\notag
\end{align}
where the component $k$ of $\tilde{\bm y}$ is $(y_d,\psi_k)_{L^2(\Omega)}$ and $\bm y_B=\bm u$. The discrete harmonic extension of $u_h$, lets call it $z_h$, can be computed solving
\begin{equation}\label{eq::comp:.he}
K_{II} \bm z_I = - K_{IB} \bm u.
\end{equation}

Using \eqref{eq::discretederivative} with $\mE_h = \mZ_h$, the extension by $0$ in the interior nodes, we can compute
\[
\functional_h'(u_h)v_h = (-\bm A^\top_{BI}\bm\varphi_I + M_{BI}\bm y_I + M_{BB}\bm u + \tikhonov (K_{BI}\bm z_I + K_{BB}\bm u) - \tilde{\bm y}_B - \tikhonov K_{B:}\tilde{\bm z})\cdot \bm v,
\]
where $\bm A^\top_{BI}$ is $(\bm A^\top)_{BI}$ and $\tilde{\bm z}$ is the discrete harmonic extension of an approximation of $u_d$; see Remark \ref{re::R7.a} below. In practice, we do a further splitting $\bm\varphi_I = \bm\phi_I + \tilde{\bm\varphi}_I$, where
\begin{align}
    \bm A^\top_{II}\bm \phi_I &= M_{II}\bm y_I + M_{IB}\bm u,\label{eq::comp::adj}\\
    \bm A^\top_{II}\tilde{\bm \varphi}_I &=  - \tilde{\bm y}_I.\notag
\end{align}
The equation $\functional'(u_h)=0$ hence reads as
\[
-\bm A^\top_{BI}\bm\phi_I + M_{BI}\bm y_I + M_{BB}\bm u + \tikhonov (K_{BI}\bm z_I + K_{BB}\bm u) =  \bm A^\top_{BI}\tilde{\bm\varphi}_I  +\tilde{\bm y}_B - \tikhonov K_{B:}\tilde{\bm z}.
\]
The left hand side is linear in $\bm u$; it is $T\bm u$. The right hand side does not depend on $\bm u$; it is $\bm w$. To compute $T\bm u$, we have to solve the three linear systems \eqref{eq::comp::st}, \eqref{eq::comp:.he} and \eqref{eq::comp::adj}. To do this efficiently, we obtain proper factorizations once using Matlab's \texttt{[L,U,P,Q,D] = lu(A(I,I))} and \texttt{[R,p] = chol(K(I,I),'vector')}. Hence, the computation of $T\bm u$ involves only the resolution of six triangular systems.
Since we have an efficient way to compute $T\bm u$, we use the preconditioned conjugate gradient method to solve the optimization problem.

The code has been done with Matlab R2025a and run on a desktop PC with 32GB of RAM and Windows 11. The meshes have been prepared using functions provided by Johannes Pfefferer. The finite element approximations are obtained with code prepared by us. 
 The optimization of the resulting finite-dimensional quadratic program is done using \texttt{pcg}. The use of graded meshes usually makes advisable the use of a preconditioner; see \cite[Example 3.5]{Mateos2018}. We have found that $\bm A_{B,B}\bm A^\top_{B,B}$ is an appropriate preconditioner for this problem, reducing the computation time for the experiments by up to a 90\% in the finest meshes.

\begin{remark}\label{re::R7.a}
    As noticed in the introductory paragraphs of Section \ref{S5.2}, we do not know any easy means to compute $\mQ_h u_d$. Nevertheless, since $u_d\in W^{3/2,2}_{\vec\beta}(\Gamma)\hookrightarrow C(\Gamma)$, we can use $\mI_h u_d$ and obtain the same order of convergence noting that
    \begin{align*}
        \langle \mD_h u_d - \mD_h\mI_h u_d,v_h\rangle_\Gamma &=  (\nabla \mH_h(u_d-\mI_h u_d),\nabla \mH_h v_h\rangle_\Gamma \\
        &\leq M_{\mH}^2 \Vert u_d-\mI_h u_d\Vert_{H^{1/2}(\Gamma)}\Vert v_h\Vert_{H^{1/2}(\Gamma)} \\
        &\leq M_{\mH}^2M_{\tr} c_{\mI} h^s\Vert u_d\Vert_{W^{3/2,2}_{\vec\beta}(\Gamma)}\Vert v_h\Vert_{H^{1/2}(\Gamma)} \ \forall v_h\in U_h.
    \end{align*}
\end{remark}

\begin{remark}\label{re::R7.1}
Computational details discussed above are similar to those discussed in \cite[Section 5.2]{GongMateosSinglerZhang2022}.
The theory discussed in \cite{OfPhanSteinbach2015} uses a continuous extension operator, but the practical implementation discussed in \cite[Section 3.5]{OfPhanSteinbach2015} uses tacitly the  discrete extension operator $\mZ_h$.
\end{remark}
\begin{remark}
    Another possibility would be to solve the complete optimality system
\[
\left(
\begin{array}{ccc}
 -M_{II}  &  \bm A^\top_{II}  & -M_{IB}\\
  \bm A_{II}  &          & \bm A_{IB} \\
  M_{BI}  & -\bm A^\top_{BI} & M_{BB} + \tikhonov D
\end{array}
\right)
\left(\begin{array}{c}
     \bm y_I  \\
     \bm\varphi_I \\
     \bm u
\end{array}
\right) = 
\left( \begin{array}{c}
     -\tilde{\bm y}_I  \\
     \\
     \tilde{\bm y}_B + \tikhonov K_{B:}\tilde z
\end{array}
\right)
\]
where $D = K_{BB} - K_{BI} K_{II}^{-1} K_{IB}$; cf \cite[eq. (3.32)]{OfPhanSteinbach2015}.

It is possible to avoid the inversion of $K_{II}$ explicitly, including the harmonic extension as an unknown. The system to solve is
\begin{equation}\label{eq::bigsystem}
\left(
\begin{array}{cccc}
K_{II} & & & K_{IB} \\
 &-M_{II}  &  \bm A^\top_{II}  & -M_{IB}\\
 & \bm A_{II}  &          & \bm A_{IB} \\
\tikhonov K_{BI} & M_{BI}  & -\bm A^\top_{BI} & M_{BB} + \tikhonov K_{BB}
\end{array}
\right)
\left(\begin{array}{c}
     \bm z_I  \\
     \bm y_I  \\
     \bm\varphi_I \\
     \bm u
\end{array}
\right) = 
\left( \begin{array}{c}
     \\
     -\tilde{\bm y}_I  \\
     \\
     \tilde{\bm y}_B + \tikhonov K_{B:}\tilde z
\end{array}
\right).
\end{equation}
In the examples described below, we have been able to solve this system using Matlab \texttt{mldivide} up to the refinement level $j=9$, obtaining exactly the same results shown in tables \ref{Table_1b} and \ref{Table_2}. The computation times are significantly bigger than the ones obtained using the preconditioned gradient method described above; see Table \ref{tab:placeholder}. 
\begin{table}[]
    \centering
    \begin{tabular}{c c c}
        $j$ &  Solving \eqref{eq::bigsystem} & Using \texttt{pcg}\\\hline
        7   & 3 s  & 3 s\\
        8   & 14 s & 11 s  \\
        9   & 150 s & 69 s \\
        10 & $\infty$ & 385 s
    \end{tabular}
    \caption{Computation times for the different solving strategies at different refinement levels.}
    \label{tab:placeholder}
\end{table}

\end{remark}

\subsection{Examples}
We present two examples. In both cases we take  the L-shaped domain $\Omega=(-1,1)^2\setminus[0,1]\times[-1,0]$, $A$ the identity matrix, and the regularization parameter $\tikhonov = 0.1$.

We use a family of graded meshes obtained by bisection. The reader is referred to \cite[Section 1.3]{AMPR2019} for a short description of the method and possible alternatives. We obtain a hierarchical family of meshes of size $h_j= 2^{-j}\sqrt{2}$, for $j=1,\ldots,10$; see Figure \ref{Figure_2} for the mesh obtained at the refinement level $j=4$. We also define $n_j = \dim(Y_{h_j}) + \dim(U_{h_j}) + \dim(Y_{0,{h_j}})$. For this kind of meshes, we have that $n_j = O(1/h_j^2)$. 

\begin{figure}
  \centering
  \includegraphics[width=.5\textwidth]{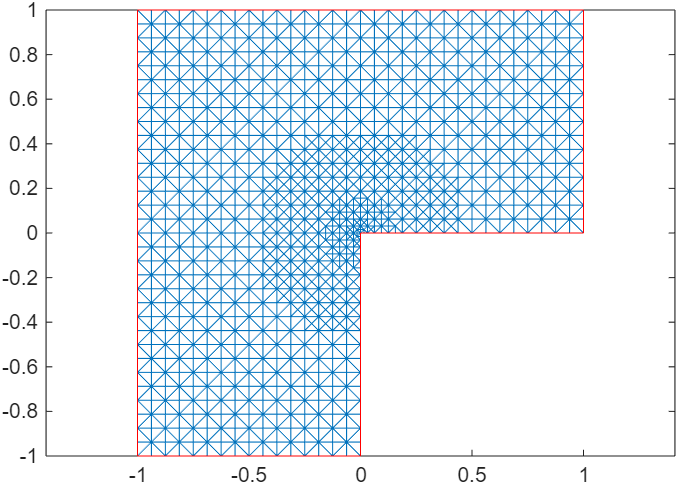}
  \caption{Mesh at the refinement level $j=4$ for $\mu=2/3$}\label{Figure_2}
\end{figure}

\begin{example}\label{ex:1}
We first test the example provided in \cite{PalGudi2026}. Although the resulting bilinear form $\mathfrak{a}(\cdot,\cdot)$ is coercive in $H^1_0(\Omega)\times H^1_0(\Omega)$, it will serve us for test purposes. In this example $b(x_1,x_2)=(-2x_1,-2x_2)$ and $a_0\equiv 4/3$.  Consider $\bar y=r^{2/3}\sin(2\theta/3)$, where $(r,\theta)$ are the polar coordinates, and $\bar\varphi = x_1^2(1-x_1^2)^2x_2^2(1-x_2^2)^2$. Finally, define $y_d = \bar y - ( -\Delta\bar\varphi -\nabla\cdot(b\bar\varphi) + a_0 \bar\varphi)$ and $u_d = \bar y_{\vert \Gamma}$. The solution of the problem is $\bar u=u_d$. 

For the grading parameter $\mu<2/3$, the order of convergence predicted by our results is $s=1$.
We report on  the experimental order of convergence 
\[s_j = -(\log_2(e_j)-\log_2(e_{j-1})),\]
where, at every mesh level, we measure the error
\[e_j = \Vert \bar y_{h_j}-\bar y\Vert_{H^1(\Omega)}+\Vert \bar u_{h_j}-\bar u\Vert_{H^{1/2}(\Gamma)}+\Vert \bar\varphi_{h_j}-\bar \varphi\Vert_ {H^1_0(\Omega)}.\]
The quantity $\Vert \bar u_h-\bar u\Vert_{H^{1/2}(\Gamma)}$ is measured by $\Vert \mH_h\bar u_h -\mH\bar u\Vert_{H^1(\Omega)}$. Since we have the exact solution, we compute the error using a Gauss formula of degree 5 with 7 nodes; see \cite{COOLS2003445}. 
The convergence history can be seen in Table \ref{Table_1} for $\mu = 0.66$ and Table \ref{Table_1b} for $\mu = 0.5$.

The performance is comparable to the one obtained by adaptative mesh refinement using the marker strategy proposed in \cite{Dorfler1996} with parameter $\theta=0.3$: the order of convergence obtained in \cite{PalGudi2026} for the error in $\Vert \bar\eta_h-\bar \eta\Vert_{H^1_0(\Omega)}+\Vert \bar z_h-\bar z\Vert_{H^1(\Omega)}+\Vert \bar \varphi_h-\bar \varphi\Vert_{H^1_0(\Omega)}$ is $N^{-1/2}$, where $N = \dim(Y_{0,h}) + \dim(Y_h) + \dim(Y_{0,h})$ is the number of degrees of freedom of the problem. Here, $\bar z$ is the optimal control variable, which coincides with $\mH \bar u$. Since $\bar y$ is harmonic, we have that $\bar\eta = \eta_{\mH \bar u} = 0$.

\begin{table}
  \centering
  \begin{tabular}{crcrcc}
 $j$ &  $\dim(Y_h)$     &     $\functional_h(\bar u_h)$   &   $n_j$  &  $e_j$   &  $s_j$ \\ \hline
 1 &      24 & 0.16105943 &      48 & \num{8.18e-01} &  \\
 2 &      78 & 0.17880768 &     156 & \num{5.78e-01} & 0.50 \\
 3 &     271 & 0.19746148 &     542 & \num{2.44e-01} & 1.24 \\
 4 &    1011 & 0.20884284 &    2022 & \num{1.12e-01} & 1.13 \\
 5 &    3885 & 0.21215951 &    7770 & \num{5.45e-02} & 1.04 \\
 6 &   15217 & 0.21301700 &   30434 & \num{2.76e-02} & 0.98 \\
 7 &   60206 & 0.21323325 &  120412 & \num{1.42e-02} & 0.96 \\
 8 &  239526 & 0.21328743 &  479052 & \num{7.36e-03} & 0.95 \\ 
 9 &  955679 & 0.21330097 & 1911358 & \num{3.83e-03} & 0.95 \\ 
 10& 3813203 & 0.21330436 & 7626406 & \num{1.98e-03} & 0.95 \\\hline
   &$\functional(\bar u)$ & 0.21330549     
 \end{tabular}
  \caption{Example \ref{ex:1}. Convergence history for $\mu=0.66$}\label{Table_1}
\end{table}

\begin{table}
  \centering
  \begin{tabular}{crcrcc}
 $j$ &  $\dim(Y_h)$     &     $\functional_h(\bar u_h)$   &   $n_j$  &  $e_j$   &  $s_j$ \\ \hline
 1 &      24 & 0.16105943 &      48 & \num{8.18e-01} &  \\
 2 &      81 & 0.17881056 &     162 & \num{5.60e-01} & 0.55 \\
 3 &     294 & 0.19746119 &     588 & \num{2.23e-01} & 1.33 \\
 4 &    1101 & 0.20884268 &    2202 & \num{9.46e-02} & 1.23 \\
 5 &    4229 & 0.21215944 &    8458 & \num{4.28e-02} & 1.14 \\
 6 &   16557 & 0.21301698 &   33114 & \num{2.03e-02} & 1.08 \\
 7 &   65502 & 0.21323324 &  131004 & \num{9.84e-03} & 1.04 \\
 8 &  260541 & 0.21328742 &  521082 & \num{4.85e-03} & 1.02 \\
 9 & 1039418 & 0.21330097 & 2078836 & \num{2.41e-03} & 1.01 \\
10 & 4152036 & 0.21330436 & 8304072 & \num{1.20e-03} & 1.00 \\\hline
   &$\functional(\bar u)$ & 0.21330549     
 \end{tabular}
  \caption{Example \ref{ex:1}. Convergence history for $\mu=0.5$}\label{Table_1b}
\end{table}
\end{example}

\begin{example} \label{ex:2}
Following \cite{AMR2024}, we take $b(x) = \delta r^\alpha(x_1,x_2)$ and $a_0(x) = r^\alpha$, where $r=(x_1^2+x_2^2)^{-1/2}$, $\delta = 6$ and $\alpha = -1.25$. Naming $\x_1=(0,0)$, these coefficients satisfy the regularity assumptions described in Remark \ref{R2.2} for any $\hat p\in(2,8)$ and $\vec\beta = (\beta_1,\ldots,\beta_6)$ with $0.25<\beta_1$ and $\beta_i=0$ for $j=2,\ldots,6$. In particular we can choose $\beta_1 > 1/3 = 1-\lambda_1$. The resulting operator $\mA_0$ is non-coercive. We fix $y_d\equiv 1$ and $u_d\equiv 0$. For the grading parameter $\mu=0.5<2/3$, the order of convergence predicted by our results is $s=1$. Since we do not have the exact solution, we determine the experimental order of convergence measuring the error between consecutive iterations as
\[\hat e_j = \Vert \bar y_{h_j}-\bar y_{h_{j+1}}\Vert_{H^1(\Omega)}+\Vert \bar u_{h_j}-\bar u_{h_{j+1}}\Vert_{H^{1/2}(\Gamma)}+\Vert \bar\varphi_{h_j}-\bar \varphi_{h_{j+1}}\Vert_ {H^1_0(\Omega)}.\]
Notice that although this procedure may yield worse error approximations, it gives a better estimate of the experimental order of convergence; see Remark \ref{re::R7.2} below. The error in the control is measured as $\Vert \bar u_{h_j}-\bar u_{h_{j+1}}\Vert_{H^{1/2}(\Gamma)} =\Vert \mH_{h_{j+1}}(\bar u_{h_j}-\bar u_{h_{j+1}})\Vert_{H^1(\Omega)}$. The convergence history can be seen in Table \ref{Table_2}.
\begin{table}
  \centering
  \begin{tabular}{crcrcc}
 $j$ &  $\dim(Y_h)$     &     $\functional_h(\bar u_h)$   &   $n_j$  &  $\hat e_j$   &  $s_j$ \\ \hline
 1 &      24 & 0.00811026 &      48 & \num{8.85e-01} &      \\ 
 2 &      81 & 0.01222197 &     162 & \num{5.65e-01} &  0.65\\ 
 3 &     294 & 0.01336101 &     588 & \num{2.95e-01} &  0.94\\ 
 4 &    1101 & 0.01365082 &    2202 & \num{1.47e-01} &  1.00\\ 
 5 &    4229 & 0.01373023 &    8458 & \num{7.32e-02} &  1.01\\ 
 6 &   16557 & 0.01375173 &   33114 & \num{3.65e-02} &  1.00\\ 
 7 &   65502 & 0.01375741 &  131004 & \num{1.82e-02} &  1.00\\ 
 8 &  260541 & 0.01375890 &  521082 & \num{9.13e-03} &  1.00\\ 
 9 & 1039418 & 0.01375928 & 2078836 & \num{4.57e-03} &  1.00\\ \hline
10 & 4152036 & 0.01375938 & 8304072 &                &      \\
  \end{tabular}
  \caption{Example \ref{ex:2}. Non-coercive operator. Convergence history for $\mu=0.5$}\label{Table_2}
\end{table}
\end{example}

\begin{remark}\label{re::R7.2}
    In general, we can prove that the order of convergence obtained using the difference between consecutive iterates is the same as the order of convergence of the original sequence. Let $s$ be a positive real number.
If we have the estimate $\Vert z_h-z\Vert \leq C h^s$, then
\[\Vert z_h-z_{h/2}\Vert \leq \Vert z_h-z\Vert +\Vert z_{h/2}-z\Vert \leq C\left(1+\frac{1}{2^s}\right) h^s.\]
On the other hand, suppose $\Vert z_h-z_{h/2}\Vert \leq C h^s$ and take a sequence $h_j = h/2^j$. Since $z_{h_j}\to z$, we can write
\[\Vert z_h-z\Vert \leq \sum_{j=0}^{\infty} \Vert z_{h_j}-z_{h_{j+1}}\Vert  \leq \sum_{j=0}^\infty C h_j^s = C \sum_{j=0}^\infty\frac{1}{2^{s j}} h^s = \frac{C}{1-\frac{1}{2^s}} h^s.\]
Therefore, the order of convergence, which is the quantity of interest, is the same. The reader is invited to experiment with the sequences $h_j = 1/2^j$ and $z_j = h_j^s$ or $z_j = (-1)^j h_j^s$ for $s=0.5$, $s = 1$ and $s = 2$ and to compare the experimental orders of convergence obtained by this method taking $\hat e_j= \Vert z_{j}-z_{j+1}\Vert $ and by the classical method of measuring the error $e_j \approx e_j^*=\Vert z_j - z_{j^*}\Vert $, where, for instance, $j^*=10$ or $j^*=11$, which are usually reference values when $h$ is the mesh size of a finite element mesh in $\mathbb{R}^2$.
\end{remark}

\appendix

\section{Some error estimates for the $H^{1/2}(\Gamma)$-projection}
\label{sec:appA}

\begin{lemma}\label{C5.4}
  If $u\in W^{3/2,2}_{\vec\beta}(\Gamma)$ with $0\leq \beta_j < 1$ for all $j\in\{1,\ldots,m\}$ and if the exponent $s$ satisfies \eqref{eq:sprime}, then there exists a constant $C>0$ independent of $u$ and $h$ such that
  \[\Vert u-\mQ_h u\Vert_{H^{1/2}(\Gamma)} \leq C h^{s}\Vert u\Vert_{W^{3/2,2}_{\vec\beta}(\Gamma)}.\]
\end{lemma}

\begin{proof}
  The conditions on $\vec\beta$ imply that $u\in H^{3/2-\beta}(\Gamma)\hookrightarrow C(\Gamma)$ with $\beta=\max_j\beta_j$, so the nodal interpolant $\mI_h u$ is well defined. Hence, from the interpolation error estimate \eqref{eq:interpolation_error_Gamma} we obtain
   \[\Vert u-\mQ_h u\Vert_{H^{1/2}(\Gamma)} \leq \Vert u-\mI_h u\Vert_{H^{1/2}(\Gamma)} \leq C h^{s}\Vert u\Vert_{W^{3/2,2}_{\vec\beta}(\Gamma)},\]
and the proof is complete.\end{proof}

\begin{lemma}\label{T5.6XXX}
  For every $u\in H^{1/2}(\Gamma)$ and all $s\leq 1$ such that $s <\lambda_{\Delta,j}/\mu_j$,  for all $j\in\{1,\ldots,m\}$
  \[\Vert u-\mQ_h u\Vert_{L^2(\Gamma)} \leq C h^{s/2} \Vert u-\mQ_h u\Vert_{H^{1/2}(\Gamma)}.\]
\end{lemma}
\begin{proof}
For every $j\in\{1,\ldots,m\}$ we can always choose a $\beta_j\in[0,1)$ such that $s< (1-\beta_j)/\mu_j <\lambda_{\Delta,j}/\mu_j$.
Consider $g\in H^{1/2}(\Gamma)\hookrightarrow W^{1/2,2}_{\vec\beta}(\Gamma)$.
  From Theorem \ref{L5.5} we deduce the existence of a unique $w_g\in H^{1/2}(\Gamma)$ solution of the dual problem
  \[\mathfrak{p}(v,w_g) = (g, v)_\Gamma\ \forall v\in H^{1/2}(\Gamma) .\]
  Furthermore,  $w_g\in W^{3/2,2}_{\vec\beta}(\Gamma)$ and
   \[\Vert w_g\Vert_{W^{3/2,2}_{\vec\beta}(\Gamma)} 
   \leq C \Vert g\Vert_{W^{1/2,2}_{\vec\beta}(\Gamma)}
   \leq C \Vert g\Vert_{H^{1/2}(\Gamma)}.\]
  Therefore, with $g=u-\mQ_hu\in H^{1/2}(\Gamma)\hookrightarrow H^{-1/2}(\Gamma)$
  \begin{align*}
    \Vert u-\mQ_h u\Vert_{L^2(\Gamma)}^2 &= \mathfrak{p}(u-\mQ_h u,w_g) =  \mathfrak{p}(u-\mQ_h u,w_g-\mI_h w_g) \\ &\leq C \Vert u-\mQ_h u\Vert_{H^{1/2}(\Gamma)} \Vert w_g-\mI_h w_g\Vert_{H^{1/2}(\Gamma)}\\
    &\leq  C\Vert u-\mQ_h u\Vert_{H^{1/2}(\Gamma)} h^{s} \Vert w_g\Vert_{W^{3/2,2}_{\vec\beta}(\Gamma)} \leq C \Vert u-\mQ_h u\Vert_{H^{1/2}(\Gamma)}^2 h^s,
  \end{align*}
  and the proof concludes taking the square root.
\end{proof}

\begin{corollary}
  If $u\in W^{3/2,2}_{\vec\beta}(\Gamma)$ for $\vec\beta$ such that $1-\lambda_{\Delta,j}< \beta_j < 1$ and $\beta_j\geq 0$ for all $j\in\{1,\ldots,m\}$, then for $s\leq 1$, $s<\lambda_{\Delta,j}/\mu_j$ for all $j\in\{1,\ldots,m\}$, there exists a constant $C>0$ independent of $u$ and $h$ such that
  \[\Vert u-\mQ_h u\Vert_{L^{2}(\Gamma)} \leq C h^{3s/2}\Vert u\Vert_{W^{3/2,2}_{\vec\beta}(\Gamma)}.\]
\end{corollary}
\begin{proof}
  The result is a straightforward consequence of lemmata \ref{C5.4} and \ref{T5.6XXX}.
\end{proof}

\section{The Deny-Lions Lemma in $H^{1/2}(\Gamma)$}
Let us recall the notation
\[\vert u \vert_{b} = \left(\int_\Gamma\int_\Gamma \frac{(u(x)-u(y))^2}{\vert x-y\vert ^{2}}\dx\mathrm{d}y\right)^{1/2}\text{ and } \Vert u \Vert_{b}^2 = \Vert u \Vert_{L^2(\Gamma)}^2 +\vert u \vert_{b}^2.\]
\begin{lemma}\label{LB1}
  There exists a constant $C_\Gamma>0$ such that
  \[\inf_{c\in \mathbb{R}}\Vert u-c \Vert_{b} \leq C_\Gamma \vert u \vert_{b}\ \forall u\in H^{1/2}(\Gamma).\]
\end{lemma}
\begin{proof}
  First we show the existence of $C_\Gamma>0$ such that
  \begin{equation}\label{EB2}
    \Vert u\Vert_{b}\leq C_{\Gamma} \left(\vert u \vert_{b} + \left\vert  \int_\Gamma u\dx\right\vert \right)\ \forall u\in H^{1/2}(\Gamma).
  \end{equation}
  Assume this is false. Then, there exists a sequence $(v_n)\subset H^{1/2}(\Gamma)$ such that 
  \begin{equation}\label{EB3}
    \Vert v_n\Vert_{b} = 1\text{ and }\lim_{n\to\infty} \left( \vert v_n \vert_{b} +  \left\vert \int_\Gamma v_n\dx\right\vert \right) = 0.
  \end{equation}
  Since $(v_n)$ is a bounded sequence, there exists $v\in H^{1/2}(\Gamma)$ and a subsequence such that $v_n\rightharpoonup v$ weakly in $H^{1/2}(\Gamma)$. From the compactness of the embedding of $H^{1/2}(\Gamma)$ in $L^2(\Gamma)$, this convergence is strong in $L^2(\Omega)$. Furthermore, from \eqref{EB3} we have that $\vert v_n \vert_{b}\to 0$. Then $(v_n)$ is a Cauchy sequence in $H^{1/2}(\Gamma)$, which is a complete metric space, and hence $(v_n)$ converges strongly in $H^{1/2}(\Gamma)$ to $v$. Therefore
  \[\int_\Gamma\int_\Gamma \frac{(v(x)-v(y))^2}{\vert x-y\vert ^{2}}\dx\mathrm{d}y = \vert v\vert _{b}^2 = \lim_{n\to\infty}\vert v_n\vert _{b}^2 =0,\]
  and hence $v$ is constant. But using the strong convergence of $v_n$ to $v$ and \eqref{EB3} we have that
  \[\int_\Gamma v\dx = \lim_{n\to\infty} \int_\Gamma v_n\dx =0.\]
  Since $v$ is constant, this implies that $v=0$, which contradicts the fact that $\Vert v_n\Vert_{b}=1$.
  
  The rest is standard. For a given $u\in H^{1/2}(\Gamma)$, define $\tilde c = \frac{1}{\vert \Gamma\vert }\int_\Gamma u \dx\in\mathbb{R}$. Applying \eqref{EB2} to $u-\tilde c$ we obtain
  \begin{align*}
  \inf_{c\in \mathbb{R}}\Vert u-c \Vert_{b} & \leq 
  \Vert u-\tilde c \Vert_{b} \leq C_\Gamma
  \left( \vert u-\tilde c \vert_{b} +  \left\vert  \int_\Gamma (u-\tilde c) \dx\right\vert \right) = C_\Gamma
   \vert u\vert_{b}
  \end{align*}
  which is the desired result.
\end{proof}


\end{document}